\documentclass[amssymb,12pt ]{article}
\pdfoutput=1
\usepackage{geometry}
\usepackage{amsfonts}
\usepackage{amssymb}
\usepackage{amsmath}
\usepackage{amsthm}
\usepackage{color}
\usepackage{graphicx}
\usepackage{picture}
\usepackage{curves}
\usepackage{subfigure}
\usepackage{epic}

\newtheorem{thm}{Theorem}[section]
\newtheorem{cor}[thm]{Corollary}
\newtheorem{lem}[thm]{Lemma}
\newtheorem{pro}[thm]{Proposition}

\newtheorem{obs}[thm]{Observation}
\newenvironment{ack}{\noindent{\bf Acknowledgments}}

\newcommand{\vol}{{\rm vol}}
\newcommand{\cs}{{\rm cs}}
\newcommand{\li}{{\rm Li}_2}
\newcommand{\imaginary}{{\rm Im}\,}

\newcommand{\modulo}{~~({\rm mod}~\pi^2)}
\newcommand{\modulos}{~~({\rm mod}~4\pi^2)}

\begin{document}

\title{Optimistic limits of colored Jones polynomials and complex volumes of hyperbolic links
}
\author{\sc Jinseok Cho}
\maketitle
\begin{abstract}
{The optimistic limit is the mathematical formulation of the classical limit which is a physical method to
expect the actual limit by using saddle point method of certain potential function.
The original optimistic limit of the Kashaev invariant was formulated by Yokota, and 
a modified formulation was suggested by the author and others. 
The modified version was easier to handle and more combinatorial than the original one.}

On the other hand, it was known that the Kashaev invariant coincides with the evaluation of the colored Jones polynomial at the certain root of unity.
{The optimistic limit of the colored Jones polynomial was also formulated by the author and others, but
it was so complicated and needed many unnatural assumptions.

In this article, we suggest a modified optimistic limit of the colored Jones polynomial, following the idea of the modified optimistic limit of the Kashaev invariant,} 
and show that it determines the complex volume of a hyperbolic link.
Furthermore, we show that this optimistic limit coincides with the optimistic limit of the Kashaev invariant modulo $4\pi^2$.
{This new version is easier to handle and more combinatorial than the old version,
and has many advantages than the modified optimistic limit of the Kashaev invariant.
Because of these advantages, several applications have already appeared and more are in preparation now.}

\end{abstract}

\section{Introduction}\label{sec1}

{For a hyperbolic link $L$, the volume conjecture, proposed by Kashaev in \cite{Kashaev97} 
 claims the following nontrivial relation:
   $$2\pi \lim_{N\rightarrow\infty}\frac{\log|\langle L\rangle_N|}{N}=\vol(L)$$
where vol($L$) is the hyperbolic volume of the link complement $\mathbb{S}^3\backslash L$
  and $\langle L\rangle_N$ is the $N$-th Kashaev invariant. This conjecture is interesting because 
  it relates geometric aspects of $L$ with the quantum invariants. Some people believes it is a hint to more deeper connection between
  geometric and quantum topology. (See \cite{McMullen11} for example.)
  After that, the generalized conjecture was proposed in \cite{Murakami02} that
 $$2\pi i\lim_{N\rightarrow\infty}\frac{\log\langle L\rangle_N}{N}\equiv i(\vol(L)+i\,\cs(L))\modulo,$$
 where cs($L$) is the Chern-Simons invariant of $\mathbb{S}^3\backslash L$ defined modulo $\pi^2$ in \cite{Zickert09}.
This conjecture is now called {\it (generalized) volume conjectures} and $\vol(L)+i\,\cs(L)$ is called {\it the complex volume} of $L$.

When the volume conjecture was first proposed in \cite{Kashaev97}, Kashaev considered {\it classical limit} of the Kashaev invariant
and verified his conjecture for three examples. Classical limit is a method of mathematical physics to expect the actual limit by using
saddle point method of certain potential function. Although the behavior of the classical limit looks very amazing, 
it cannot be well-defined due to the ambiguity of the choice of the potential function. Therefore Yokota suggested combinatorial method to
define the potential function at \cite{YokotaPre} and showed that some saddle point of his potential function determines the hyperbolic volume.
This method was first named {\it optimisitic limit} at \cite{Murakami00b} and developed by several authors at \cite{Cho09b} and \cite{Yokota10}.}

Recently, the author together with H. Kim and S. Kim suggested a modified optimistic limit 
of any link diagram at \cite{Cho13a} by using slightly different potential function. 
Comparing with previous definition in \cite{Yokota10}, this new definition was easy to handle and
had natural geometric meaning. (We will summarize the results of \cite{Cho13a} in Section \ref{sec5}.)
Furthermore, the new definition has several applications on the quantum dilogarithm function 
in \cite{Hikami07}, the quandle theory in \cite{Cho14a}
and the cluster algebra in \cite{Hikami13}. (The application to the cluster algebra of \cite{Hikami13} will be discussed in the author's later article.)

On the other hand, Kashaev invariant was proved to be the special value of the colored Jones polynomial in \cite{Murakami01a} as follows:
$$J_L(N;\exp\frac{2\pi i}{N})=\langle L \rangle_N,$$
where $\langle L \rangle_N$ is the $N$-th Kashaev invariant of a link $L$ and 
$J_L(N;x)$ is the $N$-th colored Jones polynomial of $L$ with a complex variable $x$.
The optimistic limit of the colored Jones polynomial, which uses different potential function from the Kashaev invariant version, 
was first proposed in \cite{Ohnuki05}, and developed at \cite{Cho10a} and \cite{Cho13b}. 
However, the general method developed at \cite{Cho13b} was very complicated and needed several unnatural assumptions.
In this article, we will suggest a modified optimistic limit of the colored Jones polynomial using the idea of \cite{Cho13a}. 
This modified definition, which uses slightly different potential function from \cite{Cho13b},
shares the same advantages of the definition in \cite{Cho13a}, 
namely it is easy to handle and has natural geometric meaning.

{Two optimistic limits of the Kashaev invariant and the colored Jones polynomial are almost the same in many ways.
Although they use different potential functions, which are denoted by $V(z_1,\ldots,z_g)$ and $W(w_1,\ldots,w_n)$, respectively,
and slightly different subdivisions of the same octahedral decomposition, the resulting values are the same complex volume.
(The potential function $W(w_1,\ldots,w_n)$ will be defined in Section \ref{sec2}.)
However, due to the difference of the subdivision, the colored Jones polynomial version has a wonderful advantage
that the set of equations 
\begin{equation}\label{defH}
\mathcal{I}:=\left\{\left.\exp\left(w_k\frac{\partial W}{\partial w_k}\right)=1\right|k=1,\ldots,n\right\}.
\end{equation}
always have a solution for any diagram of a hyperbolic link $L$. As a matter of fact, for any given boundary-parabolic representation
$\rho:\pi_1(L)\rightarrow {\rm PSL(2,\mathbb{C})}$ of the link group $\pi_1(L):=\pi_1(\mathbb{S}^3\backslash L)$ and for any link diagram $D$, 
we can construct a solution of $\mathcal{I}$ that induces the representation $\rho$.
(This was proved in one of the author's later article \cite{Cho14c}.) The optimistic limit of the Kashaev invariant has the same property, 
which was proved in \cite{Cho14a}, but some diagram cannot have any solution. See Figure \ref{nosolution} in Section \ref{sec5}, for an example.

The existence of a solution for any link diagram is very useful property because, by using it, we can study the hyperbolic structure
of the link combinatorially. This approach already has several interesting applications, for example, \cite{Cho14c}, \cite{Cho15b} and \cite{Cho15a}, and more applications are in preparation now.

The set of hyperbolicity equations consists of the gluing equations and the completeness condition of certain triangulation.
In the case of $\mathcal{I}$, it is related to an ideal triangulation of $\mathbb{S}^3\backslash (L\cup\{\text{two points}\})$, which will be defined in Section \ref{sec3}.
We name it {\it five-term triangulation} and the two removed points in $\mathbb{S}^3$ will be denoted by $\pm\infty$.
The exact relationship between the five-term triangulation and the set $\mathcal{I}$ is the following proposition.}

\begin{pro}\label{pro1} For a hyperbolic link $L$ with a fixed diagram, consider the potential function $W(w_1,\ldots,w_n)$ of the diagram.
Then the set $\mathcal{I}$ defined in (\ref{defH}) becomes 
the hyperbolicity equations of the five-term triangulation of $\mathbb{S}^3\backslash (L\cup\{\pm\infty\})$.
\end{pro}

We remark that Proposition \ref{pro1} was essentially proved in Section 4 of \cite{Cho13b}. 
{However, the proof in \cite{Cho13b} is very long and complicated, and what we need is only part of it, 
so we will sketch the proof of Proposition \ref{pro1} in Section \ref{sec3} for the reader's convenience.

Note that many parts of this article overlap with the author's previous article \cite{Cho13b}.
However, the major difference is that we are using triangulation of  $\mathbb{S}^3\backslash (L\cup\{\pm\infty\})$,
whereas the previous work used triangulation of $\mathbb{S}^3\backslash L$. Therefore, when we proved some properties
at \cite{Cho13b}, we first considered the general case and then proceeded to special cases 
when certain edges or faces of the triangulation are collapsed to vertices. (There were so many special cases and it required
many unnatural assumptions on link diagrams.) 
However, in this article, the proofs of the general case in \cite{Cho13b} are good enough
and this removes almost all technical difficulties of the previous work. This is why we develop this new version in this article.}

Let $\mathcal{T}=\{(w_1,\ldots,w_n)\}$ be the set of solutions\footnote{
We only consider solutions satisfying the condition that, 
when the potential function is expressed by 
$W(w_1,\ldots,w_n)=\sum\pm\li(w)+\text{(extra terms)}$,
the variables inside the dilogarithms satisfy $w\notin\{0,1,\infty\}$. 
Previously, in \cite{Yokota10} and \cite{Cho13b}, these solutions were called {\it essential solutions}.}
of $\mathcal{I}$ in $\mathbb{C}^n$. {Then, according to the result in \cite{Cho14c}, we know
$\mathcal{T}\neq \emptyset$.\footnote{The article \cite{Cho14c} depends on this article, 
so using $\mathcal{T}\neq \emptyset$ may look illogical. However, the proof of it in \cite{Cho14c} relies only on Proposition \ref{pro1} of this article and it does not require the fact $\mathcal{T}\neq \emptyset$. 
Furthermore, all results in this article still work well if we just {\it assume} $\mathcal{T}\neq \emptyset$.}}
By Theorem 1 of \cite{Tillmann11}, all edges in the five-term triangulation are essential.
(Essential edge roughly means it is not null-homotopic. See \cite{Tillmann11} for the exact definition.)
Therefore, for a solution $\bold{w}\in \mathcal{T}$, 
we can construct a boundary-parabolic representation\footnote{
The solution $\bold{w}\in\mathcal{T}$ satisfies the completeness condition, so $\rho_{\bold{w}}$ is boundary-parabolic.}
\begin{equation}\label{repre}
\rho_{\bold{w}}:\pi_1(\mathbb{S}^3\backslash (L\cup\{\pm\infty\}))=\pi_1(\mathbb{S}^3\backslash L)\longrightarrow{\rm PSL}(2,\mathbb{C}),
\end{equation}
using Yoshida's construction in Section 4.5 of \cite{Tillmann13}.
Note that the volume $\vol(\rho_{\bold{w}})$ and the Chern-Simons invariant $\cs(\rho_{\bold{w}})$
of $\rho_{\bold{w}}$ were defined in \cite{Zickert09}.
We call $\vol(\rho_{\bold{w}})+i\,\cs(\rho_{\bold{w}})$ {\it the complex volume of} $\rho_{\bold{w}}$.

For the solution set $\mathcal{T}\subset\mathbb{C}^n$, let $\mathcal{T}_{j}$ be a path component of $\mathcal{T}$ 
satisfying $\mathcal{T}=\cup_{j\in J}\mathcal{T}_{j}$ for some index set $J$. We assume $0\in J$ for notational convenience. 
To obtain well-defined values from the potential function $W(w_1,\ldots,w_n)$, we slightly modify it  to
\begin{equation}\label{defW_0}
W_0(w_1,\ldots,w_n):=W(w_1,\ldots,w_n)-\sum_{k=1}^n \left(w_k\frac{\partial W}{\partial w_k}\right)\log w_k.
\end{equation}
Then the main result of this article follows.

\begin{thm}\label{thm1}
Let $L$ be a hyperbolic link with a fixed diagram, $W(w_1,\ldots,w_n)$ be the potential function of the diagram {and
$\mathcal{T}=\cup_{j\in J}\mathcal{T}_{j}$ be the solution set of $\mathcal{I}$.}
Then, for any ${\bold w}\in \mathcal{T}_{j}$, $W_0({\bold w})$ is constant (depends only on $j$) and
\begin{equation}\label{W1}
W_0(\bold{w})\equiv i\,(\vol(\rho_{\bold w})+i\,\cs(\rho_{\bold w}))\modulo,
\end{equation}
where $\rho_{\bold w}$ is the boundary-parabolic representation obtained in (\ref{repre}).
Furthermore, there exists a path component $\mathcal{T}_{0}$ of $\mathcal{T}$ satisfying
\begin{equation}\label{W2}
W_0(\bold{w_{\infty}})\equiv i\,(\vol(L)+i\,\cs(L))\modulo,
\end{equation}
for all ${\bold w}_{\infty}\in \mathcal{T}_{0}$.
\end{thm}

The proof will be given in Section \ref{sec4}. The main idea is to use
Zickert's formula of the extended Bloch group in \cite{Zickert09}.
Although this idea was already used in \cite{Cho13a} and several others,
this proof had not appeared anywhere before. 

We call the value $W_0(\bold{w})$ {\it the optimistic limit of the colored Jones polynomial}. Note that
it depends on the choice of the diagram and the path component $\mathcal{T}_{j}$.

{ The optimistic limit of the Kashaev invariant, defined in \cite{Cho13a}, will be surveyed in Section \ref{sec5}. 
The potential function $V(z_1,\ldots,z_g)$ is defined from the diagram $D$ of the hyperbolic link $L$,
and the set of equations
\begin{equation*}
\mathcal{H}:=\left\{\left.\exp\left(z_k\frac{\partial V}{\partial z_k}\right)=1\right|k=1,\ldots,g\right\}.
\end{equation*}
becomes the hyperbolicity equation of the four-term triangulation. Four-term triangulation is obtained from the same octahedron
of the five-term triangulation by subdivding it into four tetrahedra. Therefore, four-term triangulation is a triangulation
of $\mathbb{S}^3\backslash (L\cup\{\pm\infty\})$.

Although both sets of hyperbolicity equations $\mathcal{I}$ and $\mathcal{H}$ are based on the same octahedron decomposition
of $\mathbb{S}^3\backslash (L\cup\{\pm\infty\})$, these two sets are quite different. The variables $w_1,\ldots,w_n$ of $\mathcal{I}$
are assigned to regions of the link diagram $D$, but $z_1,\ldots,z_g$ of $\mathcal{H}$ are assigned to sides of $D$. (See Figure \ref{label}.)
Furthermore, the equations in $\mathcal{I}$ are all gluing equations and they induces the completeness condition, whereas
the equations in $\mathcal{H}$ are all the completeness conditions along the meridian and they induces the gluing equations.
The author feels these two definitions of the optimistic limits seem to be dual to each other.}

\begin{figure}[h]
\centering
    \subfigure[Positive crossing]{\setlength{\unitlength}{0.5cm}
  \begin{picture}(8,5.5)\thicklines
    \put(6,5){\vector(-1,-1){4}}
    \put(2,5){\line(1,-1){1.8}}
    \put(4.2,2.8){\vector(1,-1){1.8}}
    \put(3.5,1){$w_j$}
    \put(6,2.5){$w_k$}
    \put(3.5,4.5){$w_l$}
    \put(1,2.5){$w_m$}
    \put(1,5.3){$z_d$}
    \put(6,5.3){$z_c$}
    \put(1,0.2){$z_a$}
    \put(6,0.2){$z_b$}
  \end{picture}}\hspace{2cm}
    \subfigure[Negative crossing]{\setlength{\unitlength}{0.5cm}
    \begin{picture}(8,5.5)\thicklines
    \put(2,5){\vector(1,-1){4}}
   \put(6,5){\line(-1,-1){1.8}}
   \put(3.8,2.8){\vector(-1,-1){1.8}}
    \put(3.5,1){$w_j$}
    \put(6,2.5){$w_k$}
    \put(3.5,4.5){$w_l$}
    \put(1,2.5){$w_m$}
    \put(1,5.3){$z_d$}
    \put(6,5.3){$z_c$}
    \put(1,0.2){$z_a$}
    \put(6,0.2){$z_b$}
  \end{picture}}\caption{Assignment of variables}\label{label}
\end{figure}
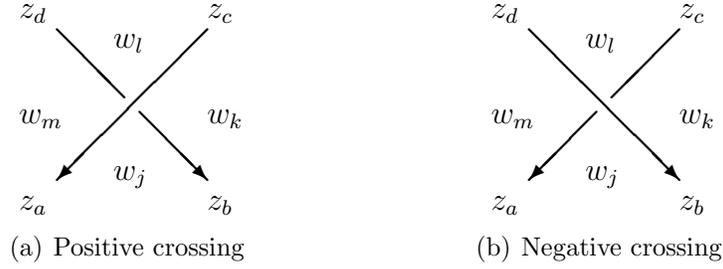

Let $\mathcal{S}=\{(z_1,\ldots,z_g)\}$ be the set of solutions of $\mathcal{H}$ in $\mathbb{C}^g$. 
Then, for a solution $\bold{z}\in \mathcal{S}$, 
we can obtain a boundary-parabolic representation
\begin{equation*}
\rho_{\bold{z}}:\pi_1(\mathbb{S}^3\backslash (L\cup\{\pm\infty\}))
=\pi_1(\mathbb{S}^3\backslash L)\longrightarrow{\rm PSL}(2,\mathbb{C}).
\end{equation*}
Now we modify the potential function $V$ to 
\begin{equation*}
V_0(z_1,\ldots,z_g):=V(z_1,\ldots,z_g)-\sum_{k=1}^g \left(z_k\frac{\partial V}{\partial z_k}\right)\log z_k.
\end{equation*}
Then the main result of \cite{Cho13a} can be summarized to the following identity:
\begin{equation}\label{V}
V_0(\bold{z})\equiv i\,(\vol(\rho_{\bold z})+i\,\cs(\rho_{\bold z}))\modulo.
\end{equation}
From (\ref{W1}) and (\ref{V}), we can easily see that, if 
$\rho_{\bold w}=\rho_{\bold z}$, then
\begin{equation}\label{coin}
W_0(\bold{w})\equiv V_0(\bold{z})\modulo.
\end{equation} 
This is formulated in stronger form as following:

\begin{thm}\label{thm2}
Assume the diagram $D$ of the hyperbolic link $L$ does not have a kink.
For a solution $\bold{w}\in\mathcal{T}$, 
if the variables $w_j,\ldots,w_m$ in Figure \ref{label} satisfy
\begin{equation*}
w_j+w_l\neq w_k+w_m
\end{equation*}
at all crossings, then there exists a solution $\bold{z}\in\mathcal{S}$ satisfying
\begin{equation}\label{coin2}
\rho_{\bold w}=\rho_{\bold z}\text{ and }W_0(\bold{w})\equiv V_0(\bold{z})\modulos.
\end{equation}
Inversely, for a solution $\bold{z}\in\mathcal{S}$, 
{ there always} exists a solution $\bold{w}\in\mathcal{T}$ satisfying (\ref{coin2}).
\end{thm}  

The proof of Theorem \ref{thm2} was essentially appeared in \cite{Cho13b}.
{However, it is based on very long and complicated calculations, and what we need here is only some parts of them.
So we will sketch the proof of Theorem \ref{thm2} in Section \ref{sec6} for the reader's convenience.}

In Section \ref{sec7}, we will apply Theorem \ref{thm2} to the example of twist knots and show several numerical calculations.
Finally, in Appendix, we will discuss {the invariance of the optimistic limit under the change of signs on the variables of the potential function. 
This property will be used in author's later article.}

\section{Potential function $W(w_1,\ldots,w_n)$}\label{sec2}

Consider a hyperbolic link $L$ and its oriented diagram $D$.
We define {\it sides} of $D$ by the arcs connecting two adjacent crossing points.\footnote{
Most people use the word {\it edge} instead of {\it side} here.
However, in this paper, we want to keep the word {\it edge} for the edge of a tetrahedron.}
For example, the diagram of the figure-eight knot $4_1$ in Figure \ref{4_1} has 8 sides.
Also we define {\it regions} of $D$ by regions surrounded by sides.
For example, the diagram in Figure \ref{4_1} has 6 regions.

\begin{figure}[h]
\centering
  \includegraphics[scale=0.7]{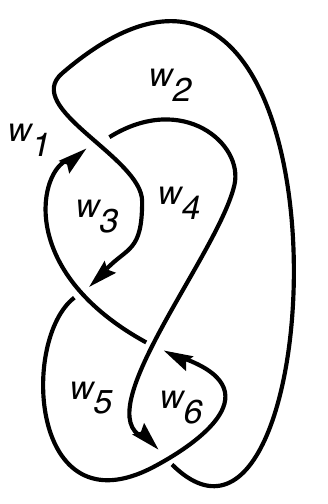}
  \caption{The figure-eight knot $4_1$}\label{4_1}
\end{figure}

We assign complex variables $w_1,\ldots,w_n$ to each region of the diagram $D$.
Using the dilogarithm function $\li(w)=-\int_0^w \frac{\log(1-t)}{t}dt$,
we define the potential function\footnote{
Note that, by using {$\approx$ to denote the equivalence relation }in Lemma 3.1 of \cite{Cho13b}, we know
$$\log\frac{w_j}{w_m}\log\frac{w_j}{w_k}\approx(\log w_j-\log w_m)(\log w_j-\log w_k)
\approx\log\frac{w_m}{w_j}\log\frac{w_k}{w_j}.$$
Therefore, changing $\log\frac{w_j}{w_m}\log\frac{w_j}{w_k}$ of $W_N$ to $\log\frac{w_m}{w_j}\log\frac{w_k}{w_j}$
does not have any effect on $\mathcal{I}$ and the optimistic limit. To avoid redundant calculation,
we will use $\log\frac{w_j}{w_m}\log\frac{w_j}{w_k}$ 
up to Section \ref{sec4} and change it to $\log\frac{w_m}{w_j}\log\frac{w_k}{w_j}$ in Section \ref{sec6}.} of a crossing as in Figure \ref{pic2}.

\begin{figure}[h]
\setlength{\unitlength}{0.4cm}
\subfigure[Positive crossing]{
  \begin{picture}(35,6)\thicklines
    \put(6,5){\vector(-1,-1){4}}
    \put(2,5){\line(1,-1){1.8}}
    \put(4.2,2.8){\vector(1,-1){1.8}}
    \put(3.5,1){$w_j$}
    \put(5.5,3){$w_k$}
    \put(3.5,4.5){$w_l$}
    \put(1.5,3){$w_m$}
    \put(8,3){$\longrightarrow$}
    \put(11,4){$W_P:=-\li(\frac{w_l}{w_m})-\li(\frac{w_l}{w_k})+\li(\frac{w_j w_l}{w_k w_m})+\li(\frac{w_m}{w_j})+\li(\frac{w_k}{w_j})$}
    \put(15,2){$-\frac{\pi^2}{6}+\log\frac{w_m}{w_j}\log\frac{w_k}{w_j}$}
  \end{picture}}\\
\subfigure[Negative crossing]{
  \begin{picture}(35,6)\thicklines
    \put(2,5){\vector(1,-1){4}}
    \put(6,5){\line(-1,-1){1.8}}
    \put(3.8,2.8){\vector(-1,-1){1.8}}
    \put(3.5,1){$w_j$}
    \put(5.5,3){$w_k$}
    \put(3.5,4.5){$w_l$}
    \put(1.5,3){$w_m$}
    \put(8,3){$\longrightarrow$}
    \put(11,4){$W_N:=\li(\frac{w_l}{w_m})+\li(\frac{w_l}{w_k})-\li(\frac{w_j w_l}{w_k w_m})-\li(\frac{w_m}{w_j})-\li(\frac{w_k}{w_j})$}
    \put(15,2){$+\frac{\pi^2}{6}-\log\frac{w_m}{w_j}\log\frac{w_k}{w_j}$}
  \end{picture}}
  \caption{Potential functions of the crossings}\label{pic2}
\end{figure}
Note that this potential function comes from the formal substitution of the R-matrix of the colored Jones polynomial.
Refer \cite{Cho13b} for details.

{\it The potential function} $W(w_1,\ldots,w_n)$ of the diagram $D$ is defined by the summation of
all potential functions of the crossings. For example, the potential function of the figure-eight knot $4_1$
in Figure \ref{4_1} is
\begin{eqnarray*}
\lefteqn{W(w_1,\ldots,w_6)}\\
&&=\left\{-\li(\frac{w_1}{w_3})-\li(\frac{w_1}{w_2})+\li(\frac{w_1 w_4}{w_2 w_3})+\li(\frac{w_3}{w_4})+\li(\frac{w_2}{w_4})
-\frac{\pi^2}{6}+\log\frac{w_3}{w_4}\log\frac{w_2}{w_4}\right\}\\
&&+\left\{-\li(\frac{w_4}{w_3})-\li(\frac{w_4}{w_5})+\li(\frac{w_1 w_4}{w_3 w_5})+\li(\frac{w_3}{w_1})+\li(\frac{w_5}{w_1})
-\frac{\pi^2}{6}+\log\frac{w_3}{w_1}\log\frac{w_5}{w_1}\right\}\\
&&+\left\{\li(\frac{w_2}{w_4})+\li(\frac{w_2}{w_6})-\li(\frac{w_2 w_5}{w_4 w_6})-\li(\frac{w_4}{w_5})-\li(\frac{w_6}{w_5})
+\frac{\pi^2}{6}-\log\frac{w_4}{w_5}\log\frac{w_6}{w_5}\right\}\\
&&+\left\{\li(\frac{w_5}{w_1})+\li(\frac{w_5}{w_6})-\li(\frac{w_2 w_5}{w_1 w_6})-\li(\frac{w_1}{w_2})-\li(\frac{w_6}{w_2})
+\frac{\pi^2}{6}-\log\frac{w_1}{w_2}\log\frac{w_6}{w_2}\right\}.
\end{eqnarray*}

We define a modified potential function $W_0(w_1,\ldots,w_n)$ as  in (\ref{defW_0}). 
Recall that $\mathcal{I}$ was defined in (\ref{defH}). 
Also recall that we are considering the solutions ${\bold w}=(w_1,\ldots,w_n)\in\mathbb{C}^n$ of $\mathcal{I}$ with the property
that if the potential function is expressed by 
$W(w_1,\ldots,w_n)=\sum\pm\li(w)+\text{(extra terms)}$,
then variables inside the dilogarithms satisfy $w\notin\{0,1,\infty\}$.

Note that the functions $\li(w)$ and $\log w$ are multi-valued functions. Therefore, to obtain well-defined values,
we have to select proper branch of the logarithm by choosing $\arg w$ and $\arg(1-w)$.
The following lemma, which corresponds to Lemma 2.1 of \cite{Cho13a},
 shows why we consider the potential function $W_0$ instead of $W$.

\begin{lem}\label{branch}
Let $\bold{w}=(w_1,\ldots,w_n)\in\mathcal{T}$. 
For the potential function $W(w_1,\ldots,w_n)$, the value of $W_0(\bold{w})$ 
is invariant under any choice of branch of the logarithm modulo $4\pi^2$.
\end{lem}

\begin{proof}
Note that it was almost proved in Lemma 2.1 of \cite{Cho13a}. 
Using the idea in \cite{Cho13a}, we can write down
\begin{equation}\label{pflem21}
W^*_0 (\bold{w})- W_0(\bold{w})\equiv
\sum_{k=1}^n \left\{-\left(w_k\frac{\partial W}{\partial w_k}\right)\log^* w_k
+\left(w_k\frac{\partial W}{\partial w_k}\right)\log w_k\right\}\modulos,
\end{equation}
where $W^*(\bold{w})$ and $\log^* w$ are the functions with different log-branch 
corresponding to an analytic continuation of $W(\bold{w})$ and $\log w$, respectively. 
We already assumed $\bold{w}=(w_1,\ldots,w_n)\in\mathcal{T}$, so we obtain
$$\left(w_k\frac{\partial W}{\partial w_k}\right)\log^* w_k
\equiv\left(w_k\frac{\partial W}{\partial w_k}\right)\log w_k\modulos,$$
and (\ref{pflem21}) is zero modulo $4\pi^2$.

\end{proof}

The following lemma and corollary were already appeared in \cite{Cho13a}
and proved as Lemma 2.2 and Corollary 2.3, respectively. 
{(Substituting $V$, $V_0$, $\mathcal{H}$, $\mathcal{S}_j$ and $z_k$ in the proof of \cite {Cho13a} to 
$W$, $W_0$, $\mathcal{I}$, $\mathcal{T}_j$ and $w_k$, respectively, gives proof.)}

\begin{lem}\label{lem1}
Let $\mathcal{T}=\cup_{j\in J}\mathcal{T}_j\subset\mathbb{C}^n$ be the solution set of $\mathcal{I}$ with $\mathcal{T}_j$ being a path component.
Assume $\mathcal{T}\neq\emptyset$.
Then, for any ${\bold w}=(w_1,\ldots,w_n)\in \mathcal{T}_j$,
$$W_0({\bold w})\equiv C_j\modulos,$$
where $C_j$ is a complex constant depending only on $j\in J$.
\end{lem}

\begin{cor}\label{lem23}
If ${\bold w}=(w_1,\ldots,w_n) \in \mathcal{T}_j$, 
then $$\lambda{\bold w}:=(\lambda w_1,\ldots,\lambda w_n) \in \mathcal{T}_j$$
for any nonzero complex number $\lambda$. Furthermore,
$$W_0({\bold w})\equiv W_0(\lambda{\bold w})\modulos.$$
\end{cor}

Due to Corollary \ref{lem23}, we can consider the solution set $\mathcal{T}$ 
as a subset of $\mathbb{CP}^{n-1}$ instead of $\mathbb{C}^n.$

\section{Five-term triangulation of $\mathbb{S}^3\backslash (L\cup\{\pm\infty\})$}\label{sec3}

In this section, we describe the five-term triangulation of $\mathbb{S}^3\backslash (L\cup\{\text{two points}\})$.
We remark that this triangulation was previously named Thurston triangulation in \cite{Cho13b}.

We place an octahedron ${\rm A}_r{\rm B}_r{\rm C}_r{\rm D}_r{\rm E}_r{\rm F}_r$ 
on each crossing $r$ of the link diagram as in Figure \ref{twistocta} 
so that the vertices ${\rm A}_r$ and ${\rm C}_r$ lie on the over-bridge and
the vertices ${\rm B}_r$ and ${\rm D}_r$ on the under-bridge of the diagram respectively. 
Then we twist the octahedron by gluing edges ${\rm B}_r{\rm F}_r$ to ${\rm D}_r{\rm F}_r$ and
${\rm A}_r{\rm E}_r$ to ${\rm C}_r{\rm E}_r$ respectively. The edges ${\rm A}_r{\rm B}_r$, ${\rm B}_r{\rm C}_r$,
${\rm C}_r{\rm D}_r$ and ${\rm D}_r{\rm A}_r$ are called {\it horizontal edges} and we sometimes express these edges
in the diagram as arcs around the crossing in the left-hand side of Figure \ref{twistocta}.

\begin{figure}[h]
\centering
\begin{picture}(6,8)  
  \setlength{\unitlength}{0.8cm}\thicklines
        \put(4,5){\arc[5](1,1){360}}
    \put(6,7){\line(-1,-1){4}}
    \put(2,7){\line(1,-1){1.8}}
    \put(4.2,4.8){\line(1,-1){1.8}}
    \put(3.7,3.2){$w_j$}
    \put(5.7,5){$w_k$}
    \put(3.7,6.7){$w_l$}
    \put(1.7,5){$w_m$}
    \put(2.2,3.9){${\rm A}_r$}
    \put(5.3,3.9){${\rm B}_r$}
    \put(5.3,5.9){${\rm C}_r$}
    \put(2.2,5.9){${\rm D}_r$}
    \put(4.3,4.8){$r$}
  \end{picture}\hspace{2cm}
\includegraphics[scale=0.5]{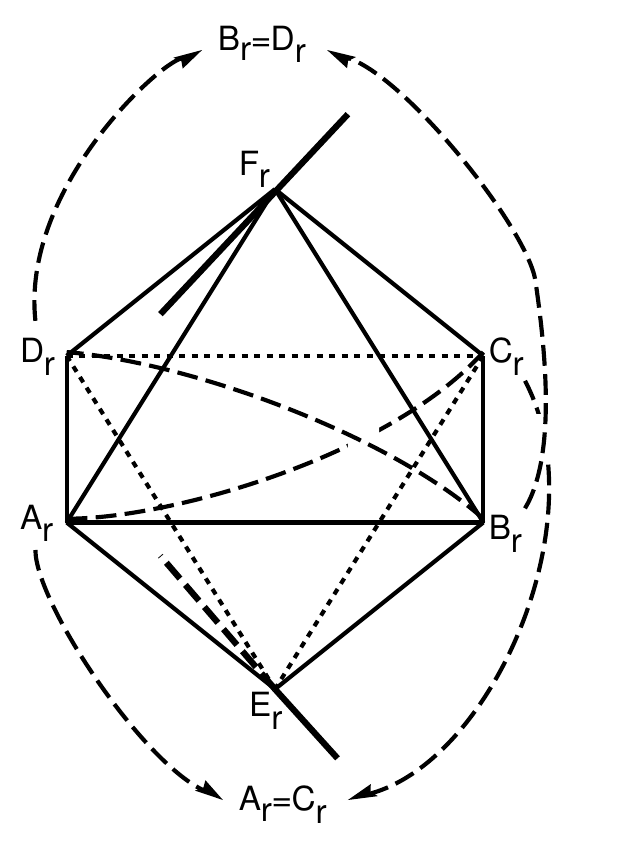}
 \caption{Octahedron on the crossing $r$}\label{twistocta}
\end{figure}
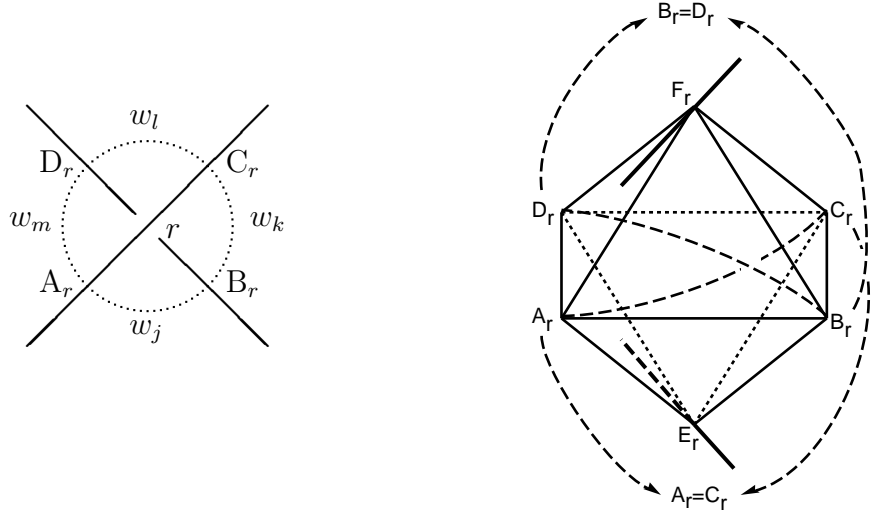

Then we glue faces of the octahedra following the sides of the diagram. 
Specifically, there are three gluing patterns as in Figure \ref{glue pattern}.
In each cases (a), (b) and (c), we identify the faces
$\triangle{\rm A}_{r}{\rm B}_{r}{\rm E}_{r}\cup\triangle{\rm C}_{r}{\rm B}_{r}{\rm E}_{r}$ to
$\triangle{\rm C}_{r+1}{\rm D}_{r+1}{\rm F}_{r+1}\cup\triangle{\rm C}_{r+1}{\rm B}_{r+1}{\rm F}_{r+1}$,
$\triangle{\rm B}_{r}{\rm C}_{r}{\rm F}_{r}\cup\triangle{\rm D}_{r}{\rm C}_{r}{\rm F}_{r}$ to
$\triangle{\rm D}_{r+1}{\rm C}_{r+1}{\rm F}_{r+1}\cup\triangle{\rm B}_{r+1}{\rm C}_{r+1}{\rm F}_{r+1}$
and
$\triangle{\rm A}_{r}{\rm B}_{r}{\rm E}_{r}\cup\triangle{\rm C}_{r}{\rm B}_{r}{\rm E}_{r}$ to
$\triangle{\rm C}_{r+1}{\rm B}_{r+1}{\rm E}_{r+1}\cup\triangle{\rm A}_{r+1}{\rm B}_{r+1}{\rm E}_{r+1}$
respectively. We call (a) {\it alternating gluing}, (b) and (c) {\it non-alternating gluings}.

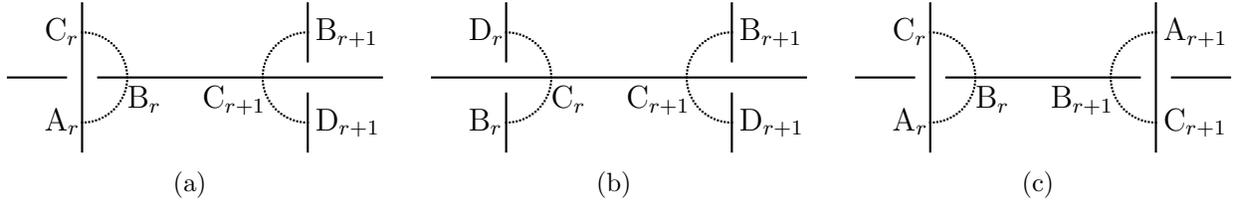
\begin{figure}[h]
\centering
  \subfigure[]
  {\begin{picture}(5,2)\thicklines
   \put(1,1){\arc[5](0,-0.6){180}}
   \put(4,1){\arc[5](0,0.6){180}}
   \put(1.2,1){\line(1,0){3.8}}
   \put(1,2){\line(0,-1){2}}
   \put(4,0){\line(0,1){0.8}}
   \put(4,2){\line(0,-1){0.8}}
   \put(0.8,1){\line(-1,0){0.8}}
   \put(0.5,0.3){${\rm A}_r$}
   \put(1.6,0.6){${\rm B}_r$}
   \put(0.5,1.5){${\rm C}_r$}
   \put(4.1,0.3){${\rm D}_{r+1}$}
   \put(2.6,0.6){${\rm C}_{r+1}$}
   \put(4.1,1.5){${\rm B}_{r+1}$}
  \end{picture}}\hspace{0.5cm}
  \subfigure[]
  {\begin{picture}(5,2)\thicklines
   \put(1,1){\arc[5](0,-0.6){180}}
   \put(4,1){\arc[5](0,0.6){180}}
   \put(5,1){\line(-1,0){5}}
   \put(1,2){\line(0,-1){0.8}}
   \put(1,0){\line(0,1){0.8}}
   \put(4,2){\line(0,-1){0.8}}
   \put(4,0){\line(0,1){0.8}}
   \put(0.5,0.3){${\rm B}_r$}
   \put(1.6,0.6){${\rm C}_r$}
   \put(0.5,1.5){${\rm D}_r$}
   \put(4.1,0.3){${\rm D}_{r+1}$}
   \put(2.6,0.6){${\rm C}_{r+1}$}
   \put(4.1,1.5){${\rm B}_{r+1}$}
  \end{picture}}\hspace{0.5cm}
  \subfigure[]
  {\begin{picture}(5,2)\thicklines
   \put(1,1){\arc[5](0,-0.6){180}}
   \put(4,1){\arc[5](0,0.6){180}}
   \put(4.2,1){\line(1,0){0.8}}
   \put(1,2){\line(0,-1){2}}
   \put(0.8,1){\line(-1,0){0.8}}
   \put(4,2){\line(0,-1){2}}
   \put(1.2,1){\line(1,0){2.6}}
   \put(0.5,0.3){${\rm A}_r$}
   \put(1.6,0.6){${\rm B}_r$}
   \put(0.5,1.5){${\rm C}_r$}
   \put(4.1,0.3){${\rm C}_{r+1}$}
   \put(2.6,0.6){${\rm B}_{r+1}$}
   \put(4.1,1.5){${\rm A}_{r+1}$}
  \end{picture}}
  \caption{Three gluing patterns}\label{glue pattern}
\end{figure}

Note that this gluing process identifies vertices $\{{\rm A}_r, {\rm C}_r\}$ to one point, denoted by $-\infty$,
and $\{{\rm B}_r, {\rm D}_r\}$ to another point, denoted by $\infty$, and finally $\{{\rm E}_r, {\rm F}_r\}$ to
the other points, denoted by ${\rm P}_j$ where $j=1,\ldots,s$ and $s$ is the number of the components of the link $L$. 
The regular neighborhoods of $-\infty$ and $\infty$ are 3-balls and that of $\cup_{j=1}^s P_j$ is
a tubular neighborhood of the link $L$. 
Therefore, if we remove the vertices ${\rm P}_1,\ldots,{\rm P}_s$ from the {octahedra}, 
then we obtain a decomposition of $\mathbb{S}^3\backslash L$, denoted by $T$.
On the other hand, if we remove all the vertices of the {octahedra}, 
the result becomes an ideal decomposition of $\mathbb{S}^3\backslash (L\cup\{\pm\infty\})$.
We call the latter {\it the octahedral decomposition} and denote it by $T'$.

To obtain an ideal triangulation from $T'$, we divide each octahedron 
${\rm A}_r{\rm B}_r{\rm C}_r{\rm D}_r{\rm E}_r{\rm F}_r$ in Figure \ref{twistocta} into five ideal tetrahedra
${\rm A}_r{\rm B}_r{\rm D}_r{\rm F}_r$, ${\rm B}_r{\rm C}_r{\rm D}_r{\rm F}_r$,
${\rm A}_r{\rm B}_r{\rm C}_r{\rm D}_r$, ${\rm A}_r{\rm B}_r{\rm C}_r{\rm E}_r$
and ${\rm A}_r{\rm C}_r{\rm D}_r{\rm E}_r$.
We call the result {\it the five-term triangulation} of $\mathbb{S}^3\backslash (L\cup\{\pm\infty\})$.
On the other hand, if we divide the same octahedron into four ideal tetrahedra
${\rm A}_r{\rm B}_r{\rm E}_r{\rm F}_r$, ${\rm B}_r{\rm C}_r{\rm E}_r{\rm F}_r$,
${\rm C}_r{\rm D}_r{\rm E}_r{\rm F}_r$ and ${\rm D}_r{\rm A}_r{\rm E}_r{\rm F}_r$,
then the result is called {\it the four-term triangulation}.
The four-term triangulation was used in \cite{Cho13a} and will appear again in Section \ref{sec5} and Section \ref{sec6} of this article.

Note that if we assign the shape parameter $u\in\mathbb{C}\backslash\{0,1\}$ to an edge of an ideal hyperbolic tetrahedron,
then the other edges are also parametrized by $u, u':=\frac{1}{1-u}$ and $u'':=1-\frac{1}{u}$
as in Figure \ref{fig6}.

\begin{figure}[h]
\begin{center}
  {\setlength{\unitlength}{0.4cm}
  \begin{picture}(12,10)\thicklines
   \put(1,1){\line(1,0){8}}
   \put(1,1){\line(1,2){4}}
   \put(5,9){\line(1,-2){4}}
   \put(9,1){\line(1,2){2}}
   \put(5,9){\line(3,-2){6}}
   \dashline{0.5}(1,1)(11,5)
   \put(5,0){$u$}
   \put(8,7){$u$}
   \put(2.2,5){$u'$}
   \put(10.5,3){$u'$}
   \put(5,3){$u''$}
   \put(7,5){$u''$}
  \end{picture}}
  \caption{Parametrization of an ideal tetrahedron with a shape parameter $u$}\label{fig6}
\end{center}
\end{figure}
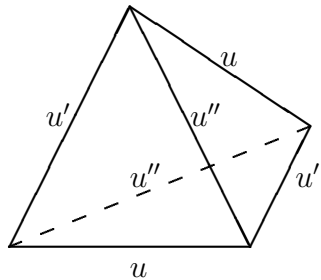
To determine the shape of the octahedron in Figure \ref{twistocta}, 
we assign shape parameters to edges of tetrahedra as in Figure \ref{fig7}. 
Note that both of $\frac{w_j w_l}{w_k w_m}$ in Figure \ref{fig7}(a) and $\frac{w_k w_m}{w_j w_l}$ in Figure \ref{fig7}(b)
are the shape parameters of the tetrahedron ${\rm A}_r{\rm B}_r{\rm C}_r{\rm D}_r$ assigned to the edges 
${\rm B}_r{\rm D}_r$ and ${\rm A}_r{\rm C}_r$. Also note that the assignment of shape parameters here does not
depend on the orientations of the link diagram.

\begin{figure}
\centering
  \subfigure[Positive crossing]
  {\begin{picture}(6,7)  
  \setlength{\unitlength}{0.8cm}\thicklines
        \put(4,5){\arc[5](1,1){360}}
    \put(6,7){\vector(-1,-1){4}}
    \put(2,7){\line(1,-1){1.8}}
    \put(4.2,4.8){\vector(1,-1){1.8}}
    \put(3.7,3.2){$w_j$}
    \put(5.7,5){$w_k$}
    \put(3.7,6.7){$w_l$}
    \put(1.7,5){$w_m$}
    \put(2.2,3.9){${\rm A}_r$}
    \put(5.3,3.9){${\rm B}_r$}
    \put(5.3,5.9){${\rm C}_r$}
    \put(2.2,5.9){${\rm D}_r$}
    \put(4.3,4.8){$r$}
  \end{picture}\hspace{2cm}\includegraphics[scale=0.6]{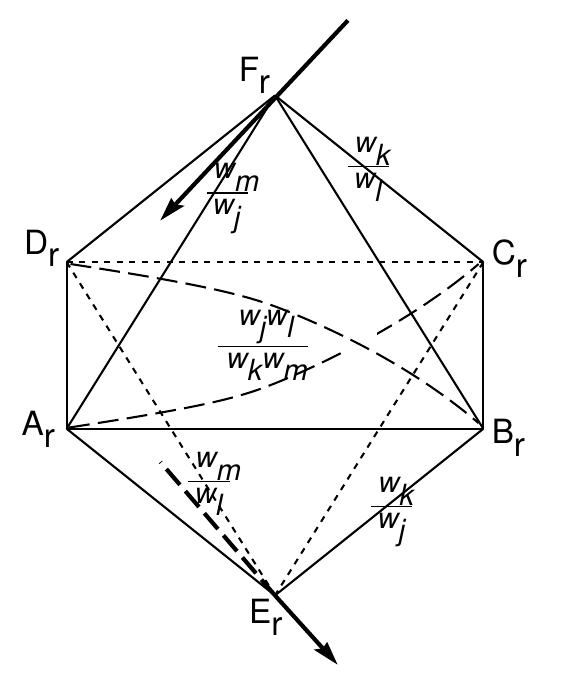}}\\
  \subfigure[Negative crossing]
  {\begin{picture}(6,7)  
  \setlength{\unitlength}{0.8cm}\thicklines
        \put(4,5){\arc[5](1,1){360}}
    \put(2,7){\vector(1,-1){4}}
   \put(6,7){\line(-1,-1){1.8}}
   \put(3.8,4.8){\vector(-1,-1){1.8}}
    \put(3.7,3.2){$w_j$}
    \put(5.7,5){$w_k$}
    \put(3.7,6.7){$w_l$}
    \put(1.7,5){$w_m$}
    \put(2.2,3.9){${\rm B}_r$}
    \put(5.3,3.9){${\rm A}_r$}
    \put(5.3,5.9){${\rm D}_r$}
    \put(2.2,5.9){${\rm C}_r$}
    \put(4.3,4.8){$r$}
  \end{picture}\hspace{2cm}\includegraphics[scale=0.6]{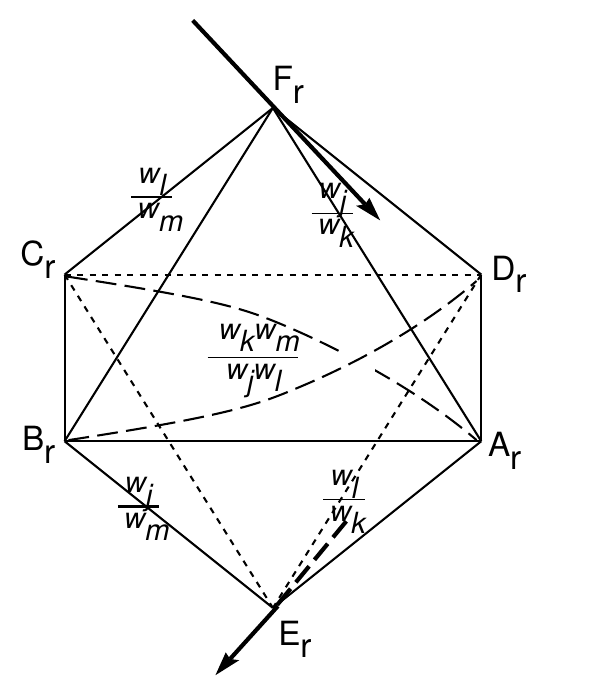}}
  \caption{Assignment of shape parameters}\label{fig7}
\end{figure}

To obtain the boundary parabolic representation 
$\pi_1(\mathbb{S}^3\backslash (L\cup\{\pm\infty\}))\longrightarrow{\rm PSL}(2,\mathbb{C})$,
we require two conditions on the ideal triangulation of $\mathbb{S}^3\backslash (L\cup\{\pm\infty\})$;
the product of shape parameters on any edge in the triangulation becomes one, and the holonomies induced
by meridian and longitude of the torus cusps act as translations on the torus cusp.
Note that these conditions are expressed as equations of shape parameters.
The former equations are called {\it (Thurston's) gluing equations}, the latter is called {\it completeness condition},
and the whole set of these equations are called {\it the hyperbolicity equations}.
As already discussed in \cite{Cho13a} and Section \ref{sec1}, a solution ${\bold w}$ of the hyperbolicity equation determines
a boundary-parabolic representation 
$$\rho_{\bold{w}}:\pi_1(\mathbb{S}^3\backslash (L\cup\{\pm\infty\}))
=\pi_1(\mathbb{S}^3\backslash L)\longrightarrow{\rm PSL}(2,\mathbb{C}).$$

The rest of this section is devoted to the proof of Proposition \ref{pro1}. 
It was already proved\footnote{
The proof is in Lemma 4.1 and Proposition 1.1 of \cite{Cho13b}, which started with the general case, 
and then proceeded to the collapsed cases.
In this article, the collapsed cases do not happen, so the general case is enough.}
in \cite{Cho13b}, so we sketch the proof here.

\begin{proof}[Sketch of the proof of Proposition \ref{pro1}]
For all the crossings of the link diagram and the corresponding octahedra in Figure \ref{twistocta}, 
let $\mathcal{A}$ be the set of horizontal edges
${\rm A}_r{\rm B}_r$, ${\rm B}_r{\rm C}_r$, ${\rm C}_r{\rm D}_r$ and ${\rm D}_r{\rm A}_r$. 
Let $\mathcal{B}$ be the set of edges ${\rm B}_r{\rm F}_r$, ${\rm D}_r{\rm F}_r$,
${\rm A}_r{\rm E}_r$, ${\rm C}_r{\rm E}_r$ of all crossings and other edges glued to them. Let $\mathcal{C}$ be
the set of edges ${\rm A}_r{\rm C}_r$ and ${\rm B}_r{\rm D}_r$ of all crossings. 
Finally, let $\mathcal{D}$ be the set of all the other edges in the triangulation. Note that
if the link diagram is alternating, then $\mathcal{D}=\emptyset$.

{Recall that $W_P$ and $W_N$ are the potential functions defined in Figure \ref{pic2}.}
Direct calculation shows that $\exp(w_a\frac{\partial W_P}{\partial w_a})$ for $a=j,k,l,m$
is the product of the shape parameters assigned to the horizontal edge 
corresponding to the region $w_a$ in Figure \ref{fig7}(a). For example, 
$$\exp\left(w_j\frac{\partial W_P}{\partial w_j}\right)
=\left(\frac{w_j w_l}{w_k w_m}\right)'\left(\frac{w_m}{w_j}\right)''\left(\frac{w_k}{w_j}\right)'',$$
which is the product of the shape parameters assigned to the edge ${\rm A}_r{\rm B}_r$.
(See (8)--(10) of \cite{Cho13b} for the other equations. 
In \cite{Cho13b}, our $W_P$ and $W_N$ were denoted by $P_1$ and $N_1$, respectively.)
Furthermore, the same holds for $\exp(w_a\frac{\partial W_N}{\partial w_a})$ too.
Therefore, $\mathcal{I}$ becomes the gluing equations of the edges in $\mathcal{A}$.

The gluing equations of the edges in $\mathcal{C}$ and $\mathcal{D}$ hold trivially
because of the assigning rule of the shape parameters to the tetrahedra. 
We will show that the gluing equations of the edges in $\mathcal{B}$ hold trivially too. 
Consider the alternating gluing in Figure \ref{fig8}(a). 
This induces a part of the cusp diagram as in Figure \ref{fig8}(b), which comes from the gluing of two tetrahedra in Figure \ref{fig8}(c).
On the other hand, the non-alternating gluings in Figure \ref{glue pattern}(b) and (c) 
do not have any effect on the cusp diagram of the torus cusp.

\begin{figure}[h]\centering
\subfigure[Gluing diagram]{
\begin{picture}(5,2)\setlength{\unitlength}{1.2cm}\thicklines
   \put(1,1){\arc[5](0,-0.6){180}}
   \put(4,1){\arc[5](0,0.6){180}}
   \put(1.2,1){\line(1,0){3.8}}
   \put(1,2){\line(0,-1){2}}
   \put(4,0){\line(0,1){0.8}}
   \put(4,2){\line(0,-1){0.8}}
   \put(0.8,1){\line(-1,0){0.8}}
   \put(0.5,0.3){${\rm A}_r$}
   \put(1.6,0.6){${\rm B}_r$}
   \put(0.5,1.5){${\rm C}_r$}
   \put(4.1,0.3){${\rm D}_{r+1}$}
   \put(2.6,0.6){${\rm C}_{r+1}$}
   \put(4.1,1.5){${\rm B}_{r+1}$}
   \put(2.2,0.3){$w_a$}
   \put(2.2,1.5){$w_b$}\end{picture}}\hspace{2cm}
      \subfigure[Cusp diagram]{\includegraphics[scale=0.65]{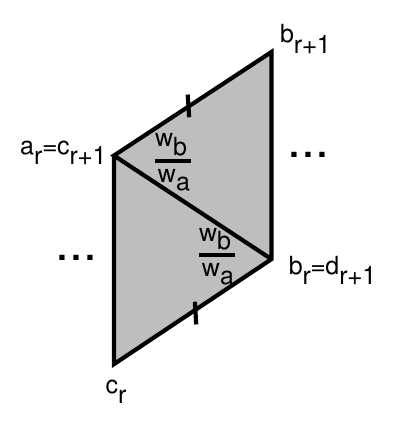}}
   \subfigure[Gluing tetrahedra]{\includegraphics[scale=0.6]{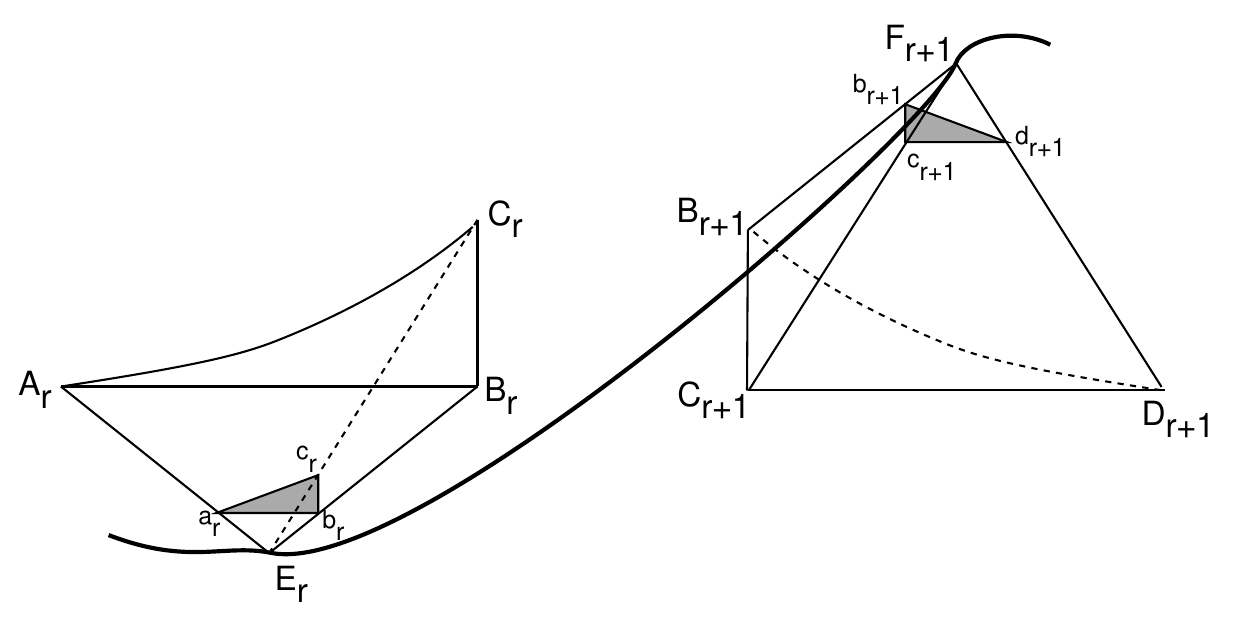}}
   \caption{Cusp diagram induced from Figure \ref{fig6}(a)}\label{fig8}
   \end{figure}
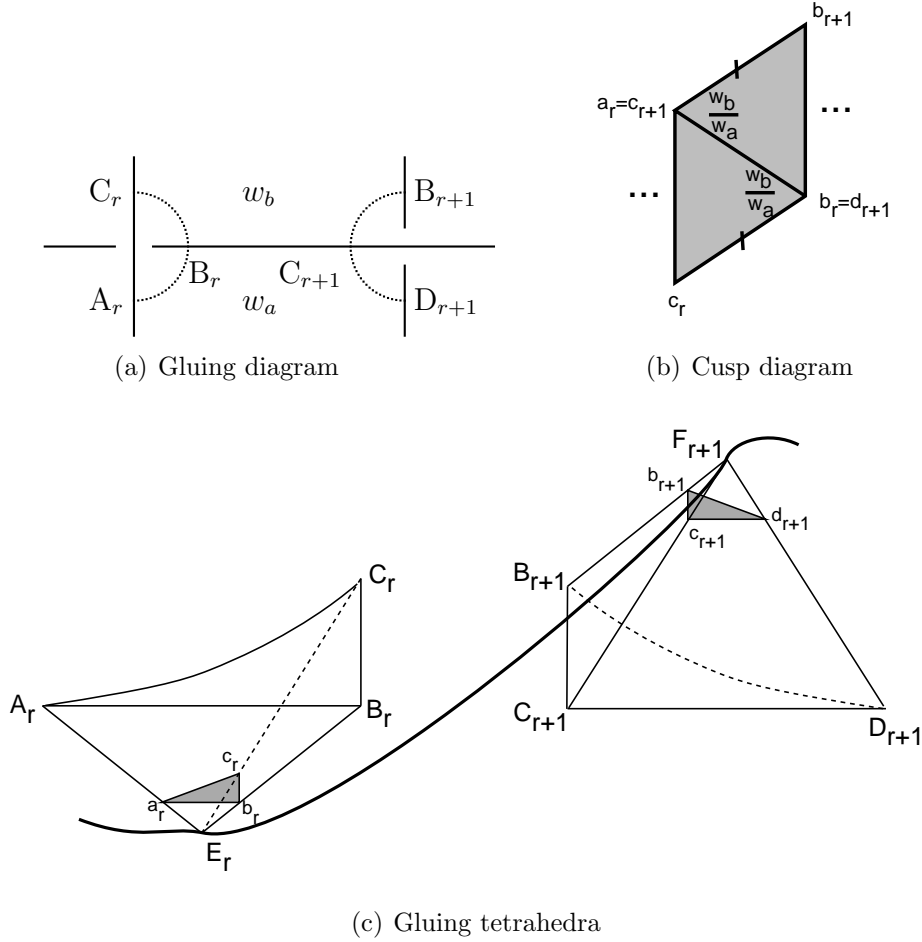

Note that the cusp diagram in Figure \ref{fig8}(b) is an annulus because the edge $c_{r+1}b_{r+1}$ is identified with $c_r b_r$.
The shape parameter $\frac{w_b}{w_a}$ is assigned to the edges 
${\rm B}_r{\rm E}_r$ and ${\rm C}_{r+1}{\rm F}_{r+1}$ in Figure \ref{fig8}(c), 
and the product of shape parameters on the edge 
${\rm B}_r{\rm E}_r={\rm B}_{r+1}{\rm F}_{r+1}={\rm D}_{r+1}{\rm F}_{r+1}\in\mathcal{B}$
(around the vertex $b_r=b_{r+1}=d_{r+1}$ in Figure \ref{fig8}(b)) is
$$\frac{w_b}{w_a}\left(\frac{w_b}{w_a}\right)''\left(\frac{w_b}{w_a}\right)'=-1.$$
Therefore, if we consider another annulus on the right-hand side of the edge $b_{r+1}d_{r+1}$ in Figure \ref{fig8}(b),
the gluing equation of the edge ${\rm B}_r{\rm E}_r={\rm B}_{r+1}{\rm F}_{r+1}={\rm D}_{r+1}{\rm F}_{r+1}\in\mathcal{B}$
is satisfied trivially. 

The other gluing equations of the edges in $\mathcal{B}$ can be obtained in the same way.
Hence, we conclude $\mathcal{I}$ induces the gluing equations of all the edges in 
$\mathcal{A}\cup\mathcal{B}\cup\mathcal{C}\cup\mathcal{D}$.
Furthermore, the cusp diagram in Figure \ref{fig8}(b) already satisfies one completeness condition of
the meridian that sends the edge $c_r b_r$ to $c_{r+1}b_{r+1}$. 
Therefore, $\mathcal{I}$ induces all the hyperbolicity equations.

\end{proof}

\section{Proof of Theorem \ref{thm1}}\label{sec4}
In this section, we always assume ${\bold w}=(w_1,\ldots,w_n)$ is a solution in $\mathcal{T}$.
The main technique of the proof of Theorem \ref{thm1} is the extended Bloch group theory in \cite{Zickert09}.
To apply it, we first define the vertex ordering of the five-term triangulation in Figure \ref{fig9} 
so that the order 0, 1, 2, 3 is assigned
to the vertices of each tetrahedra in the order of ${\rm D}_r{\rm B}_r{\rm F}_r{\rm A}_r$, 
${\rm B}_r{\rm E}_r{\rm A}_r{\rm C}_r$, ${\rm D}_r{\rm B}_r{\rm F}_r{\rm C}_r$,
${\rm D}_r{\rm E}_r{\rm A}_r{\rm C}_r$ and ${\rm D}_r{\rm B}_r{\rm A}_r{\rm C}_r$.

\begin{figure}[h]\centering
\subfigure[Positive crossing]{\includegraphics[scale=0.6]{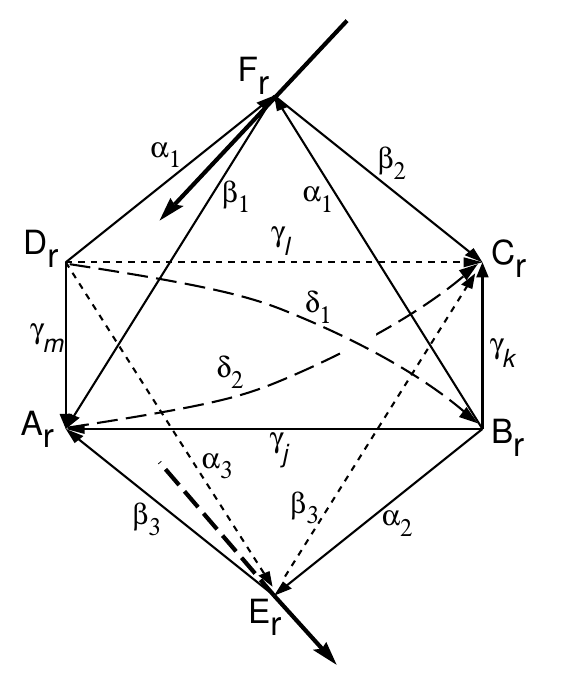}}
\subfigure[Negative crossing]{\includegraphics[scale=0.6]{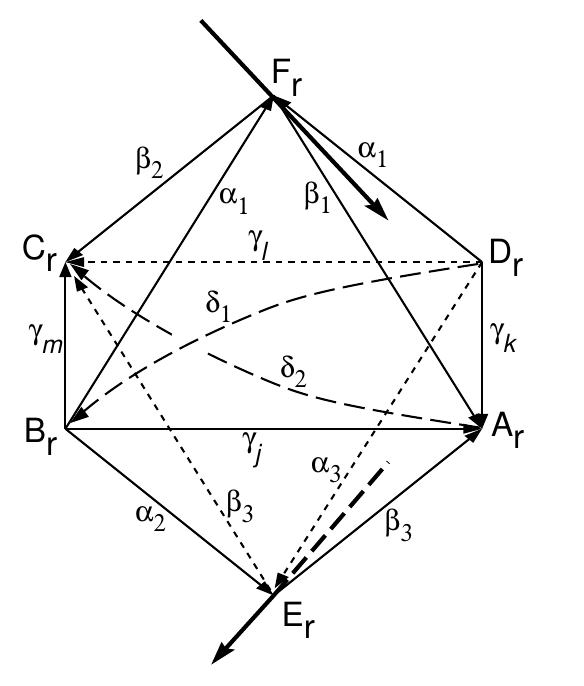}}
\caption{Vertex orderings and labelings of edges}\label{fig9}
\end{figure}

Note that the vertex ordering of each tetrahedron induces
the orientations of the edges and the tetrahedron. The induced orientation of the tetrahedron
can be different from the original orientation induced by the triangulation.
For example, {this is the case for the tetrahedra ${\rm D}_r{\rm B}_r{\rm F}_r{\rm C}_r$ and ${\rm D}_r{\rm E}_r{\rm A}_r{\rm C}_r$
in Figure \ref{fig9}(a), ${\rm D}_r{\rm B}_r{\rm F}_r{\rm A}_r$, 
${\rm B}_r{\rm E}_r{\rm A}_r{\rm C}_r$ and ${\rm D}_r{\rm B}_r{\rm A}_r{\rm C}_r$
in Figure \ref{fig9}(b).}
If the two orientations are the same, we define the sign of the tetrahedron $\sigma=1$, and if they are different,
then $\sigma=-1$.

One important property of our vertex orientation is that 
when two edges are glued together in the triangulation, 
the orientations of the two edges induced by each vertex orderings coincide.
(We call this condition {\it edge-orientation consistency}.)
Because of this property, we can apply the formula in \cite{Zickert09}.

The five-term triangulation we are using is an ideal triangulation, so we parametrized all ideal tetrahedra of the triangulation
by assigning shape parameters as in Figure \ref{fig7}.
For each tetrahedron with the vertex-orientation, we define an element of the extended pre-Bloch group
$\sigma[u^{\sigma};p,q]\in\widehat{\mathcal{P}}(\mathbb{C})$, where $\sigma$ is the sign of the tetrahedron, 
$u$ is the shape parameter assigned to the edge connecting the $0$th and $1$st vertices, and $p,q$ are certain integers.

Zickert suggested a way to determine $p$ and $q$ from the developing map 
of the representation $\rho:\pi_1(M)\rightarrow{\rm PSL}(2,\mathbb{C})$ of a hyperbolic manifold $M$ in \cite{Zickert09},
and showed that
\begin{equation}\label{formula}
\sum\sigma\widehat{L}([u^{\sigma};p,q])\equiv i(\vol(\rho)+i\,\cs(\rho))\modulo,
\end{equation}
where the summation is over all tetrahedra and 
$$\widehat{L}([u;p,q])=\li(u)-\frac{\pi^2}{6}+\frac{1}{2}q\pi i\log u+\frac{1}{2}(p\pi i+\log u)\log(1-u)$$
is a complex valued function defined on $\widehat{\mathcal{P}}(\mathbb{C})$.

{ Although our five-term triangulation is for $\mathbb{S}^3\backslash (L\cup\{\pm\infty\})$, 
the formula of \cite{Zickert09} is still valid because we can consider the two points $\pm\infty$
the interior points of the manifold $\mathbb{S}^3\backslash L$. To apply the formula, we have to remove the interior vertices,
which results in our five-term triangulation of $\mathbb{S}^3\backslash (L\cup\{\pm\infty\})$.
(See Theorem 4.11 of \cite{Zickert09} for details.)

Here, we remark that the author made a mistake in his previous article \cite{Cho13a} when justifying the usage of the triangulation of
$\mathbb{S}^3\backslash (L\cup\{\pm\infty\})$.
He mentioned the Thurston's spinning construction of \cite{Tillmann13}, 
but it can be applied when a boundary-parabolic representation is already given,
and the construction shows that the parameter space determines the volume of the representation, not the complex volume.
(Note that Lemma 2.3 and Proposition 3.1 of \cite{Tillmann13} are still valid 
for {\it any} boundary-parabolic representation and its volume.
{However, we cannot directly guarantee the invariance of the Chern-Simons invariant from \cite{Tillmann13}.})}

To determine $p, q$ of $\sigma[u^{\sigma};p,q]$ corresponding to each tetrahedron with vertex orientation, 
we assign certain complex numbers 
$g_{jk}$ to the edge connecting the $j$th and $k$th vertices, where $j,k\in\{0,1,2,3\}$ and $j<k$.
We assume $g_{jk}$ satisfies the property that 
if two edges are glued together in the triangulation, then the assigned $g_{jk}$'s of the glued edges coincide.
We do not use the exact values of $g_{jk}$ in this article, but remark that there is an explicit method in \cite{Zickert09} 
for calculating these numbers using the developing map. With the given numbers $g_{jk}$, we can calculate $p,q$ using the following equation,
which appeared as equation (3.5) in \cite{Zickert09}:
  \begin{eqnarray}
      \left\{\begin{array}{ll}
      p\pi i =-\log u^{\sigma}+\log g_{03}+\log g_{12}-\log g_{02}-\log g_{13},\\
      q\pi i =\log(1-u^{\sigma})+\log g_{02}+\log g_{13}-\log g_{01}-\log g_{23}.
      \end{array}\right.\label{pq}
  \end{eqnarray}

To avoid confusion, we use variables $\alpha_1, \alpha_2, \alpha_3, \beta_1, \beta_2, \beta_3,
\gamma_j, \gamma_k, \gamma_l, \gamma_m, \delta_1$ and $\delta_2$ instead of $g_{jk}$ as in Figure \ref{fig9}.
Note that $\gamma_a$ ($a=j,k,l,m$) is assigned to the horizontal edge that lies in
the region with $w_a$.
The orientation we defined in Figure \ref{fig9} satisfies the edge-orientation consistency,
so we will apply the formula of \cite{Zickert09} to our five-term triangulation.

For the positive crossing $r$ in Figure \ref{fig9}(a), 
let $\sigma_1^{(r)}[u_1^{\sigma_1^{(r)}};p_1^{(r)},q_1^{(r)}]$, $\ldots$, 
$\sigma_5^{(r)}[u_5^{\sigma_5^{(r)}};p_5^{(r)},q_5^{(r)}]$ be the elements in $\widehat{\mathcal{P}}(\mathbb{C})$
corresponding to ${\rm D}_r{\rm B}_r{\rm F}_r{\rm A}_r$, 
${\rm B}_r{\rm E}_r{\rm A}_r{\rm C}_r$, ${\rm D}_r{\rm B}_r{\rm F}_r{\rm C}_r$,
${\rm D}_r{\rm E}_r{\rm A}_r{\rm C}_r$ and ${\rm D}_r{\rm B}_r{\rm A}_r{\rm C}_r$ respectively. Then we have

$$\sigma_1^{(r)}=\sigma_2^{(r)}=\sigma_5^{(r)}=1,~\sigma_3^{(r)}=\sigma_4^{(r)}=-1,$$
$$u_1^{\sigma_1^{(r)}}=\frac{w_m}{w_j},~ u_2^{\sigma_2^{(r)}}=\frac{w_k}{w_j},~
u_3^{\sigma_3^{(r)}}=\frac{w_l}{w_k},~ u_4^{\sigma_4^{(r)}}=\frac{w_l}{w_m},~
u_5^{\sigma_5^{(r)}}=\frac{w_j w_l}{w_k w_m},$$
and direct calculation from (\ref{pq}) shows
\begin{eqnarray}
      \left\{\begin{array}{l}
p_1^{(r)}\pi i+\log\frac{w_m}{w_j}=\log\gamma_m-\log\gamma_j,\\
p_2^{(r)}\pi i+\log\frac{w_k}{w_j}=\log\gamma_k-\log\gamma_j,\\
p_3^{(r)}\pi i+\log\frac{w_l}{w_k}=\log\gamma_l-\log\gamma_k,\\
p_4^{(r)}\pi i+\log\frac{w_l}{w_m}=\log\gamma_l-\log\gamma_m,\\
p_5^{(r)}\pi i+\log\frac{w_j w_l}{w_k w_m}=\log\gamma_j+\log\gamma_l-\log\gamma_k-\log\gamma_m,
      \end{array}\right.\label{p1}
\end{eqnarray}
and
\begin{eqnarray}
      \left\{\begin{array}{l}
q_1^{(r)}\pi i-\log(1-\frac{w_m}{w_j})=\log\alpha_1+\log\gamma_j-\log\delta_1-\log\beta_1,\\
q_2^{(r)}\pi i-\log(1-\frac{w_k}{w_j})=\log\gamma_j+\log\beta_3-\log\alpha_2-\log\delta_2,\\
q_3^{(r)}\pi i-\log(1-\frac{w_l}{w_k})=\log\alpha_1+\log\gamma_k-\log\delta_1-\log\beta_2,\\
q_4^{(r)}\pi i-\log(1-\frac{w_l}{w_m})=\log\gamma_m+\log\beta_3-\log\alpha_3-\log\delta_2,\\
q_5^{(r)}\pi i-\log(1-\frac{w_j w_l}{w_k w_m})=\log\gamma_m+\log\gamma_k-\log\delta_1-\log\delta_2.
      \end{array}\right.\label{q1}
\end{eqnarray}

For the negative crossing $r$ in Figure \ref{fig9}(b), 
let $\sigma_1^{(r)}[u_1^{\sigma_1^{(r)}};p_1^{(r)},q_1^{(r)}]$, $\ldots$, 
$\sigma_5^{(r)}[u_5^{\sigma_5^{(r)}};p_5^{(r)},q_5^{(r)}]$ be the elements in $\widehat{\mathcal{P}}(\mathbb{C})$
corresponding to ${\rm B}_r{\rm E}_r{\rm A}_r{\rm C}_r$, ${\rm D}_r{\rm B}_r{\rm F}_r{\rm A}_r$, 
${\rm D}_r{\rm E}_r{\rm A}_r{\rm C}_r$, ${\rm D}_r{\rm B}_r{\rm F}_r{\rm C}_r$
and ${\rm D}_r{\rm B}_r{\rm A}_r{\rm C}_r$ respectively. Then we have
$$\sigma_1^{(r)}=\sigma_2^{(r)}=\sigma_5^{(r)}=-1,~\sigma_3^{(r)}=\sigma_4^{(r)}=1,$$
$$u_1^{\sigma_1^{(r)}}=\frac{w_m}{w_j},~ u_2^{\sigma_2^{(r)}}=\frac{w_k}{w_j},~
u_3^{\sigma_3^{(r)}}=\frac{w_l}{w_k},~ u_4^{\sigma_4^{(r)}}=\frac{w_l}{w_m},~
u_5^{\sigma_5^{(r)}}=\frac{w_j w_l}{w_k w_m},$$
and direct calculation from (\ref{pq}) shows
\begin{eqnarray}
      \left\{\begin{array}{l}
p_1^{(r)}\pi i+\log\frac{w_m}{w_j}=\log\gamma_m-\log\gamma_j,\\
p_2^{(r)}\pi i+\log\frac{w_k}{w_j}=\log\gamma_k-\log\gamma_j,\\
p_3^{(r)}\pi i+\log\frac{w_l}{w_k}=\log\gamma_l-\log\gamma_k,\\
p_4^{(r)}\pi i+\log\frac{w_l}{w_m}=\log\gamma_l-\log\gamma_m,\\
p_5^{(r)}\pi i+\log\frac{w_j w_l}{w_k w_m}=\log\gamma_j+\log\gamma_l-\log\gamma_k-\log\gamma_m,
      \end{array}\right.\label{p2}
\end{eqnarray}
and
\begin{eqnarray}
      \left\{\begin{array}{l}
q_1^{(r)}\pi i-\log(1-\frac{w_m}{w_j})=\log\gamma_j+\log\beta_3-\log\alpha_2-\log\delta_2,\\
q_2^{(r)}\pi i-\log(1-\frac{w_k}{w_j})=\log\alpha_1+\log\gamma_j-\log\delta_1-\log\beta_1,\\
q_3^{(r)}\pi i-\log(1-\frac{w_l}{w_k})=\log\gamma_k+\log\beta_3-\log\alpha_3-\log\delta_2,\\
q_4^{(r)}\pi i-\log(1-\frac{w_l}{w_m})=\log\alpha_1+\log\gamma_m-\log\delta_1-\log\beta_2,\\
q_5^{(r)}\pi i-\log(1-\frac{w_j w_l}{w_k w_m})=\log\gamma_k+\log\gamma_m-\log\delta_1-\log\delta_2.
      \end{array}\right.\label{q2}
\end{eqnarray}

From the above definitions, we can conclude 
$$\sum_{r\text{: crossings}}\sum_{c=1}^5\sigma_c^{(r)}[u_c^{\sigma_c^{(r)}};p_c^{(r)},q_c^{(r)}]
\in\widehat{\mathcal{P}}(\mathbb{C})$$
is the corresponding element of the five-term triangulation.
The following observation can be easily obtained.

\begin{obs}\label{obs1} There exists a constant $C$ satisfying
$$\log w_b\equiv\log \gamma_b+C~~({\rm mod}~\pi i),$$
for all $b=1,\ldots,n$.
\end{obs}

\begin{proof} The relation (\ref{p1}) or (\ref{p2}) holds for any crossing $r$ of the link diagram.
Therefore, by letting $C=\log w_1-\log\gamma_1$, it follows trivially.
\end{proof}

Now we define integer $Q_a^{(r)}$ for the crossing $r$ and $a=j,k,l,m$ by the following ways. 
For the positive crossing $r$ in Figure \ref{fig9}(a), we define
\begin{eqnarray}
      \left\{\begin{array}{l}
Q_j^{(r)}=q_1^{(r)}+q_2^{(r)}-q_5^{(r)}+p_1^{(r)}+p_2^{(r)},\\
Q_k^{(r)}=-q_2^{(r)}-q_3^{(r)}+q_5^{(r)}-p_1^{(r)},\\
Q_l^{(r)}=q_3^{(r)}+q_4^{(r)}-q_5^{(r)},\\
Q_m^{(r)}=-q_4^{(r)}-q_1^{(r)}+q_5^{(r)}-p_2^{(r)},
      \end{array}\right.\label{Q1}
\end{eqnarray}
and, for the negative crossing $r$ in Figure \ref{fig9}(b), we define
\begin{eqnarray}
      \left\{\begin{array}{l}
Q_j^{(r)}=-q_1^{(r)}-q_2^{(r)}+q_5^{(r)}-p_1^{(r)}-p_2^{(r)},\\
Q_k^{(r)}=q_2^{(r)}+q_3^{(r)}-q_5^{(r)}+p_1^{(r)},\\
Q_l^{(r)}=-q_3^{(r)}-q_4^{(r)}+q_5^{(r)},\\
Q_m^{(r)}=q_4^{(r)}+q_1^{(r)}-q_5^{(r)}+p_2^{(r)}.
      \end{array}\right.\label{Q2}
\end{eqnarray}

Note that, from the definitions (\ref{Q1}) and (\ref{Q2}), we can directly obtain
\begin{equation}\label{sumQ}
\sum_{a=j,k,l,m}Q_a^{(r)}=0,
\end{equation}
for any crossing $r$.

\begin{lem}\label{lem42} For the potential function $W(w_1,\ldots,w_n)$ and the index $b=1,\ldots,n$, we have
$$w_b\frac{\partial W}{\partial w_b}=\sum_r Q_b^{(r)}\pi i,$$
where $r$ is over the crossings that lie on the boundary of the region associated with $w_b$.
\end{lem}

\begin{proof}
Note that $W_P$ and $W_N$ were defined in Figure \ref{pic2}.

For the positive crossing $r$ in Figure \ref{fig9}(a), direct calculation from (\ref{p1}) and (\ref{q1}) shows
\begin{eqnarray*}
      \left\{\begin{array}{l}
w_j\frac{\partial W_P}{\partial w_j}=Q_j^{(r)}\pi i+(\log\beta_1-\log\alpha_1)+(\log\alpha_2-\log\beta_3),\\
w_k\frac{\partial W_P}{\partial w_k}=Q_k^{(r)}\pi i+(\log\alpha_1-\log\beta_2)+(\log\beta_3-\log\alpha_2),\\
w_l\frac{\partial W_P}{\partial w_l}=Q_l^{(r)}\pi i+(\log\beta_2-\log\alpha_1)+(\log\alpha_3-\log\beta_3),\\
w_m\frac{\partial W_P}{\partial w_m}=Q_m^{(r)}\pi i+(\log\alpha_1-\log\beta_1)+(\log\beta_3-\log\alpha_3).
      \end{array}\right.
\end{eqnarray*}
For the negative crossing $r$ in Figure \ref{fig9}(b), direct calculation from (\ref{p2}) and (\ref{q2}) shows
\begin{eqnarray*}
      \left\{\begin{array}{l}
w_j\frac{\partial W_N}{\partial w_j}=Q_j^{(r)}\pi i+(\log\alpha_1-\log\beta_1)+(\log\beta_3-\log\alpha_2),\\
w_k\frac{\partial W_N}{\partial w_k}=Q_k^{(r)}\pi i+(\log\beta_1-\log\alpha_1)+(\log\alpha_3-\log\beta_3),\\
w_l\frac{\partial W_N}{\partial w_l}=Q_l^{(r)}\pi i+(\log\alpha_1-\log\beta_2)+(\log\beta_3-\log\alpha_3),\\
w_m\frac{\partial W_N}{\partial w_m}=Q_m^{(r)}\pi i+(\log\beta_2-\log\alpha_1)+(\log\alpha_2-\log\beta_3).
      \end{array}\right.
\end{eqnarray*}

From the above calculations, we can find a general rule. 
Elaborating on $w_j\frac{\partial W_P}{\partial w_j}$, 
consider the faces ${\rm A}_r{\rm B}_r{\rm F}_r$ and ${\rm A}_r{\rm B}_r{\rm E}_r$ in Figure \ref{fig9}(a).
The term $(\log\beta_1-\log\alpha_1)$ in $w_j\frac{\partial W_P}{\partial w_j}$ comes from
the edges ${\rm A}_r{\rm F}_r$ and ${\rm B}_r{\rm F}_r$ of the face ${\rm A}_r{\rm B}_r{\rm F}_r$ counterclockwise,
and the term $(\log\alpha_2-\log\beta_3)$ comes from the edges ${\rm B}_r{\rm E}_r$ and ${\rm A}_r{\rm E}_r$
of the face ${\rm A}_r{\rm B}_r{\rm E}_r$ clockwise. These rules hold for all the cases.

Consider the face ${\rm A}_r{\rm B}_r{\rm F}_r$ and its corresponding term $(\log\beta_1-\log\alpha_1)$.
As in Figure \ref{fig10}, the face glued to ${\rm A}_r{\rm B}_r{\rm F}_r$ induces the term $(\log\alpha_1-\log\beta_1)$,
which cancel out the term corresponding to ${\rm A}_r{\rm B}_r{\rm F}_r$. 
(The shaded faces in Figure \ref{fig10} are glued to ${\rm A}_r{\rm B}_r{\rm F}_r$.) 
In the same way, all the other terms corresponding to the other faces are cancelled each other and the proof follows.

\begin{figure}[h]
\centering
\includegraphics[scale=0.5]{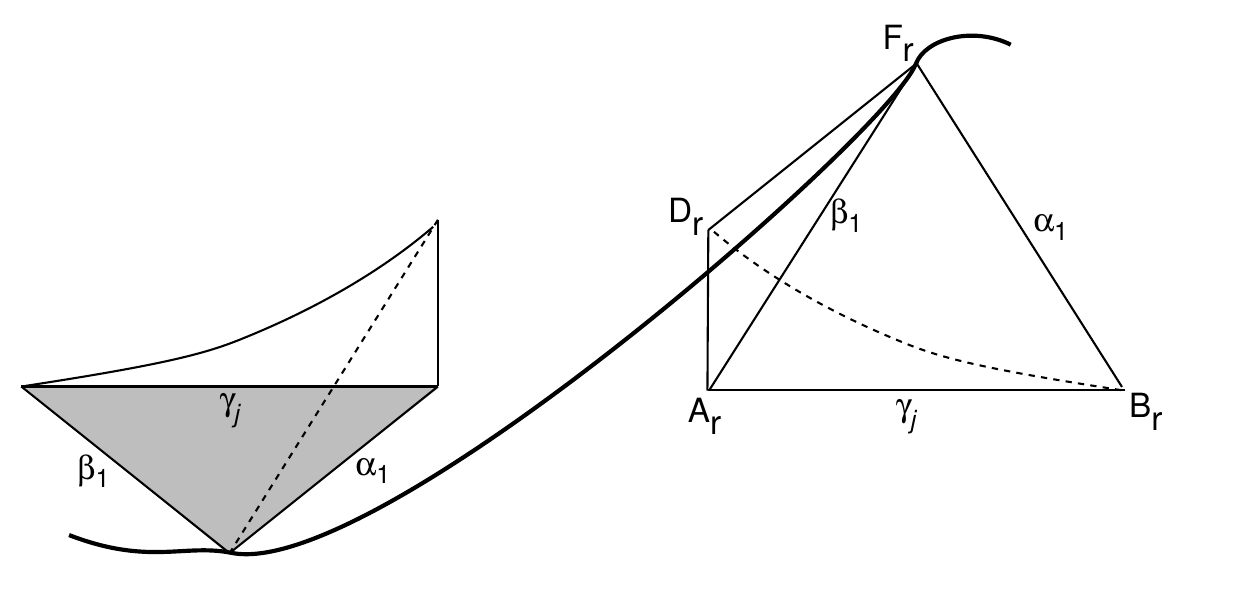}\\
\includegraphics[scale=0.5]{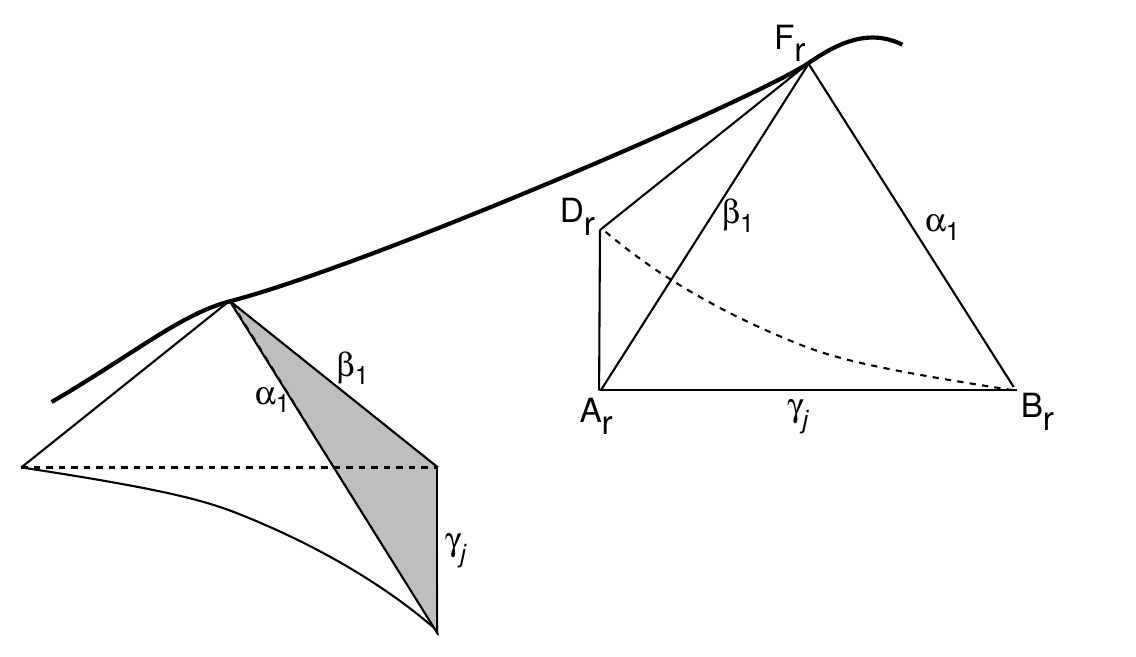}
\caption{Two cases of the gluing of ${\rm A}_r{\rm B}_r{\rm F}_r$}\label{fig10}
\end{figure}

\end{proof}

By combining (\ref{sumQ}) and Lemma \ref{lem42}, or by direct calculation, we have
\begin{equation}\label{eq20}
\sum_{b=1}^n w_b\frac{\partial W}{\partial w_b}=0.
\end{equation}

To obtain (\ref{W1}), we need to use (\ref{formula}) and prove
\begin{equation}\label{goal1}
W(w_1,\ldots,w_n)-\sum_{b=1}^n \left(w_b\frac{\partial W}{\partial w_b}\right)\log w_b\equiv
\sum_{r,c}\sigma_c^{(r)}\widehat{L}\left([u_c^{\sigma_c^{(r)}};p_c^{(r)},q_c^{(r)}]\right)\modulo,
\end{equation}
where $c=1,\ldots,5$ and $r$ is over all crossings. At first, from (\ref{Q1}) and (\ref{Q2}), we have
\begin{eqnarray}
\sum_{a=j,k,l,m}Q_a^{(r)}\pi i \log w_a\equiv -\sigma_1^{(r)}
\left\{q_1^{(r)}\pi i \log\frac{w_m}{w_j}+q_2^{(r)}\pi i \log\frac{w_k}{w_j}-q_3^{(r)}\pi i \log\frac{w_l}{w_k}\right.\nonumber\\
\left.-q_4^{(r)}\pi i \log\frac{w_l}{w_m}+q_5^{(r)}\pi i \log\frac{w_j w_l}{w_k w_m}
+p_1^{(r)}\pi i \log\frac{w_k}{w_j}+p_2^{(r)}\pi i \log\frac{w_m}{w_j}\right\}\nonumber\\
\equiv-\sum_{c=1}^5\sigma_c^{(r)}q_c^{(r)}\pi i\log u_c^{\sigma_c^{(r)}}
-\sigma_1^{(r)}p_1^{(r)}\pi i\log u_2^{\sigma_2^{(r)}}
-\sigma_1^{(r)}p_2^{(r)}\pi i\log u_1^{\sigma_1^{(r)}}~~({\rm mod}~2\pi^2).\label{eq22}
\end{eqnarray}
Combining (\ref{eq22}) and Lemma \ref{lem42}, we obtain
\begin{eqnarray}
\lefteqn{\frac{1}{2}\sum_{r,c}\sigma_c^{(r)}q_c^{(r)}\pi i\log u_c^{\sigma_c^{(r)}}
\equiv-\frac{1}{2}\sum_{b=1}^n\left(w_b\frac{\partial W}{\partial w_b}\right)\log w_b}\nonumber\\
&&-\frac{1}{2}\sum_r\left\{\sigma_1^{(r)}p_1^{(r)}\pi i\log u_2^{\sigma_2^{(r)}}
+\sigma_1^{(r)}p_2^{(r)}\pi i\log u_1^{\sigma_1^{(r)}}\right\}
\modulo,\label{qterm}\end{eqnarray}
where $c=1,\ldots,5$ and $r$ is over all crossings.

Let $W^{(r)}$ be the potential function of the crossing $r$, i.e.
\begin{eqnarray*}
W^{(r)}:=\left\{\begin{array}{l} 
W_P\text{ if }r\text{ is a positive crossing,}\\
W_N\text{ if }r\text{ is a negative crossing.}\end{array}
\right.
\end{eqnarray*}
From (\ref{p1}), (\ref{p2}) and direct calculation, we obtain
\begin{eqnarray}
\lefteqn{\sum_{c=1}^5 \sigma_c^{(r)}(p_c^{(r)}\pi i+\log u_c^{\sigma_c^{(r)}})\log(1-u_c^{\sigma_c^{(r)}})}\nonumber\\
&&=\sigma_1^{(r)}\left\{(\log\gamma_m-\log\gamma_j)\log(1-\frac{w_m}{w_j})
+(\log\gamma_k-\log\gamma_j)\log(1-\frac{w_k}{w_j})\right.\nonumber\\
&&~~-(\log\gamma_l-\log\gamma_k)\log(1-\frac{w_l}{w_k})
-(\log\gamma_l-\log\gamma_m)\log(1-\frac{w_l}{w_m})\nonumber\\
&&\left.~~+(\log\gamma_j+\log\gamma_l-\log\gamma_k-\log\gamma_m)\log(1-\frac{w_j w_l}{w_k w_m})\right\}\nonumber\\
&&=-\sum_{a=j,k,l,m}\log\gamma_a\left(w_a\frac{\partial W^{(r)}}{\partial w_a}\right)\nonumber\\
&&~~+\sigma_1^{(r)}(\log\gamma_m-\log\gamma_j)\log\frac{w_k}{w_j}
+\sigma_1^{(r)}(\log\gamma_k-\log\gamma_j)\log\frac{w_m}{w_j}
\nonumber\\
&&=-\sum_{a=j,k,l,m}\log\gamma_a\left(w_a\frac{\partial W^{(r)}}{\partial w_a}\right)\nonumber\\
&&~~+\sigma_1^{(r)}p_1^{(r)}\pi i\log u_2^{\sigma_2^{(r)}}+\sigma_1^{(r)}p_2^{(r)}\pi i\log u_1^{\sigma_1^{(r)}}
+2\log\frac{w_k}{w_j}\log\frac{w_m}{w_j}.\label{pterm}\end{eqnarray}
Using Observation \ref{obs1}, (\ref{eq20}) and 
$$w_b\frac{\partial W}{\partial w_b}\equiv 0~~({\rm mod}~2\pi i),$$
we obtain
\begin{eqnarray}
\sum_{r\text{ : crossings}}\sum_{a=j,k,l,m}\log\gamma_a\left(w_a\frac{\partial W^{(r)}}{\partial w_a}\right)
=\sum_{b=1}^n\log\gamma_b\left(w_b\frac{\partial W}{\partial w_b}\right)\nonumber\\
\equiv\sum_{b=1}^n\left(w_b\frac{\partial W}{\partial w_b}\right)\log w_b~~~~({\rm mod}~2\pi^2).\label{eq24}
\end{eqnarray}

From (\ref{qterm}), (\ref{pterm}) and (\ref{eq24}), we have
\begin{eqnarray}\label{fterm1}
\frac{1}{2}\sum_{r,c}\sigma_c^{(r)}\left\{q_c^{(r)}\pi i\log u_c^{\sigma_c^{(r)}}
+(p_c^{(r)}\pi i+\log u_c^{\sigma_c^{(r)}})\log(1-u_c^{\sigma_c^{(r)}})\right\}\nonumber\\
\equiv-\sum_{b=1}^n\left(w_b\frac{\partial W}{\partial w_b}\right)\log w_b
+\sum_r\log u_1^{\sigma_1^{(r)}}\log u_2^{\sigma_2^{(r)}}\modulo,
\end{eqnarray}
where $c=1,\ldots,5$ and $r$ is over all crossings.

By definition, the potential function $W(w_1,\ldots,w_n)$ is expressed by
\begin{equation}\label{fterm2}
W(w_1,\ldots,w_n)=\sum_{r,c}\sigma_c^{(r)}\left\{\li(u_c^{\sigma_c^{(r)}})-\frac{\pi^2}{6}\right\}
+\sum_r\log u_1^{\sigma_1^{(r)}}\log u_2^{\sigma_2^{(r)}}.
\end{equation}
From (\ref{fterm1}) and (\ref{fterm2}), we obtain (\ref{goal1}) 
and complete the proof of the first part of Theorem \ref{thm1}.

On the other hand, the existence of ${\bold w}_{\infty}$ is guaranteed by \cite{Tillmann13}. (See \cite{Cho13a} for details. 
{Or, if we allow the construction in \cite{Cho14c}, we can construct ${\bold w}_{\infty}$ from the discrete faithful representation
$\rho:\pi_1(L)\rightarrow {\rm PSL}(2,\mathbb{C})$.})
Then we can choose $\mathcal{T}_{0}$ the path component containing ${\bold w}_{\infty}$.
This completes the proof of Theorem \ref{thm1}.

\section{The optimistic limit of the Kashaev invariant}\label{sec5}

To prove Theorem \ref{thm2}, we briefly review the results of \cite{Cho13a}.

Consider a hyperbolic link $L$ and its non-oriented diagram $D$. 
(If $D$ already has an orientation, then we ignore it.)
Assume $D$ does not have any kinks\footnote{
This assumption is only for the optimistic limit of the Kashaev invariant. If the diagram has a kink, 
then the hyperbolicity equations in $\mathcal{H}$ defined in (\ref{defH2}) do not have any solution. 
On the other hand, the hyperbolicity equations in $\mathcal{I}$ always have a solution whether it has a kink or not.} 
by removing them as in Figure \ref{kink}.

We assign complex variables $z_1,\ldots,z_g$ to sides of the diagram.
Then we define the potential function of the crossing as in Figure \ref{potentialV}.

\begin{figure}[h]
\centering
  \includegraphics[scale=0.3]{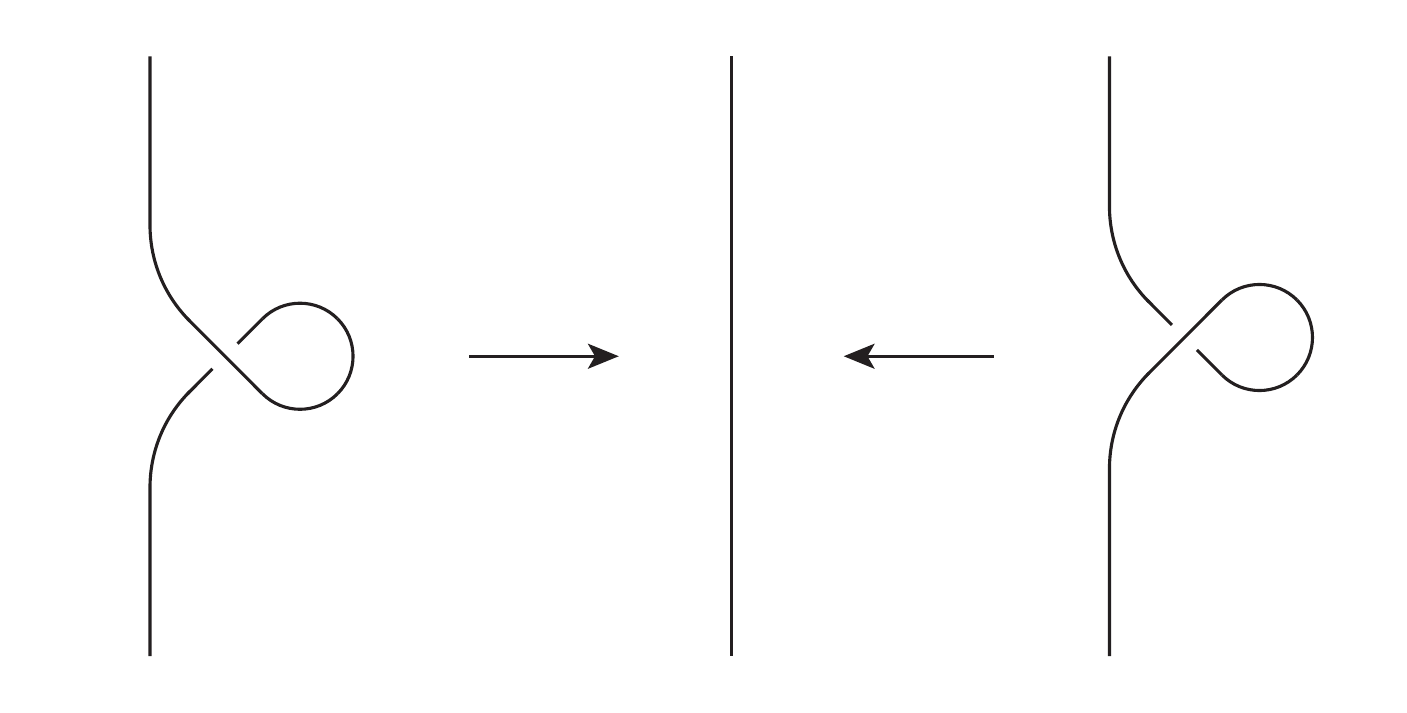}
  \caption{Removing kinks}\label{kink}
\end{figure}

\begin{figure}[h]\centering
{\setlength{\unitlength}{0.4cm}
  \begin{picture}(30,6)\thicklines
    \put(6,5){\line(-1,-1){4}}
    \put(2,5){\line(1,-1){1.8}}
    \put(4.2,2.8){\line(1,-1){1.8}}
    \put(1,5.3){$z_d$}
    \put(6,5.3){$z_c$}
    \put(1,0.2){$z_a$}
    \put(6,0.2){$z_b$}
    \put(7,3){$\displaystyle\longrightarrow ~~\li(\frac{z_b}{z_a})-\li(\frac{z_b}{z_c})+\li(\frac{z_d}{z_c})-\li(\frac{z_d}{z_a})$}
  \end{picture}}
  \caption{Potential function of a crossing}\label{potentialV}
\end{figure}

The potential function $V(z_1,\ldots,z_g)$ of the diagram $D$ is defined by the summation of all potential functions of the crossings.
Then we define the set $\mathcal{H}$ by
\begin{equation}\label{defH2}
\mathcal{H}:=\left\{\left.\exp\left(z_k\frac{\partial V}{\partial z_k}\right)=1\right|k=1,\ldots,g\right\}.
\end{equation}

Let $\mathcal{S}=\{(z_1,\ldots,z_g)\}$ be the set of solutions\footnote{
As already mentioned in Section \ref{sec1}, we only consider solutions satisfying the condition that,
when the potential function $V$ is expressed by $V(z_1,\ldots,z_g)=\sum\pm\li(\frac{z_a}{z_b})$,
the variable inside the dilogarithms satisfy $\frac{z_a}{z_b}\notin\{0,1,\infty\}$.
{ Furthermore, for the crossing in Figure \ref{parameterizing},
the solution should satisfy $\frac{z_c}{z_a}\neq 1$ and $\frac{z_d}{z_b}\neq 1$. The later condition, which the author missed in
his previous paper \cite{Cho13a}, is needed to avoid
the holonomies induced by the meridians becoming the trivial map.}} of $\mathcal{H}$ in $\mathbb{C}^g$.
We always assume $\mathcal{S}\neq\emptyset$.
Note that we cannot avoid this assumption because, if the diagram contains the left-hand side of Figure \ref{nosolution},
then $\mathcal{S}=\emptyset$, {but $\mathcal{T}\neq\emptyset$}. (See \cite{Cho13a} {and \cite{Cho14c}} for details.) 


\begin{figure}[h]
\centering
  \includegraphics[scale=0.8]{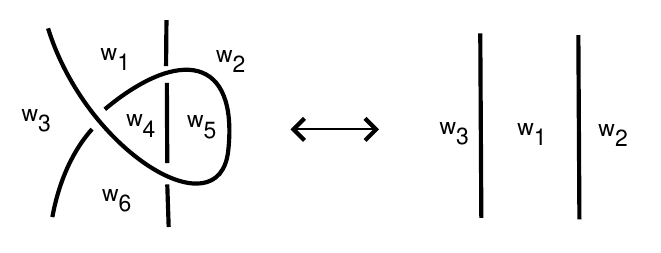}
  \caption{Diagram with $\mathcal{S}=\emptyset$ and $\mathcal{T}\neq\emptyset$}\label{nosolution}
\end{figure}

Recall the four-term triangulation of $\mathbb{S}^3\backslash (L\cup\{\pm\infty\})$ was defined in Section \ref{sec2}.
To determine the shape of tetrahedra, we assign shape parameters $\frac{z_b}{z_a}$, $\frac{z_c}{z_b}$,
$\frac{z_d}{z_c}$ and $\frac{z_a}{z_d}$ to the horizontal edges
${\rm A}_k{\rm B}_k$, ${\rm B}_k{\rm C}_k$, ${\rm C}_k{\rm D}_k$ and ${\rm D}_k{\rm A}_k$ respectively.
(See Figure \ref{parameterizing}.)
Then we obtain the following proposition, which was Proposition 1.1 of \cite{Cho13a}.

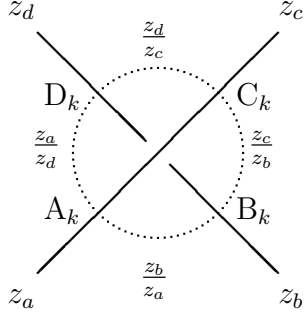
\begin{figure}[h]\centering
\begin{picture}(6,5)  
  \setlength{\unitlength}{0.8cm}\thicklines
        \put(4,3){\arc[5](1,1){360}}
    \put(6,5){\line(-1,-1){4}}
    \put(2,5){\line(1,-1){1.8}}
    \put(4.2,2.8){\line(1,-1){1.8}}
    \put(3.7,0.8){$\frac{z_b}{z_a}$}
    \put(5.5,3){$\frac{z_c}{z_b}$}
    \put(3.7,4.8){$\frac{z_d}{z_c}$}
    \put(1.9,3){$\frac{z_a}{z_d}$}
    \put(1.5,5.3){$z_d$}
    \put(6,5.3){$z_c$}
    \put(1.5,0.5){$z_a$}
    \put(6,0.5){$z_b$}
    \put(2.1,1.9){${\rm A}_k$}
    \put(5.3,1.9){${\rm B}_k$}
    \put(5.3,3.8){${\rm C}_k$}
    \put(2.1,3.8){${\rm D}_k$}
  \end{picture}\caption{Parametrizing tetrahedra}\label{parameterizing}
  \end{figure}

\begin{pro} For a hyperbolic link $L$ with a fixed diagram, 
consider the potential function $V(z_1,\ldots,z_g)$ of the diagram.
Then the set $\mathcal{H}$ defined in (\ref{defH2}) becomes 
the hyperbolicity equations of the four-term triangulation of $\mathbb{S}^3\backslash (L\cup\{\pm\infty\})$.
\end{pro}

By using Yoshida's construction in Section 4.5 of \cite{Tillmann13}, for a solution 
${\bold z}=(z_1,\ldots,z_g)\in\mathcal{S}$,
we can obtain a boundary-parabolic representation
\begin{equation}\label{rhoz}
\rho_{\bold{z}}:\pi_1(\mathbb{S}^3\backslash (L\cup\{\pm\infty\}))=\pi_1(\mathbb{S}^3\backslash L)\longrightarrow{\rm PSL}(2,\mathbb{C}).
\end{equation}

For the solution set $\mathcal{S}$, let $\mathcal{S}_j$ be a path component of $\mathcal{S}$ 
satisfying $\mathcal{S}=\cup_{j\in J'}\mathcal{S}_j$ for some index set $J'$. We assume $0\in J'$ for notational convenience. 
To obtain well-defined values from the potential function $V(z_1,\ldots,z_g)$, we slightly modify it  to
\begin{equation*}
V_0(z_1,\ldots,z_g):=V(z_1,\ldots,z_g)-\sum_{k=1}^n \left(z_k\frac{\partial V}{\partial z_k}\right)\log z_k.
\end{equation*}
Then we obtain the main result of \cite{Cho13a} as follows:

\begin{thm}\label{thm3}
Let $L$ be a hyperbolic link with a fixed diagram and  $V(z_1,\ldots,z_g)$ be the potential function of the diagram. 
Assume the solution set $\mathcal{S}=\cup_{j\in J'}\mathcal{S}_j$ is not empty.
Then, for any ${\bold z}\in \mathcal{S}_j$, $V_0({\bold z})$ is constant (depends only on $j$) and
\begin{equation*}
V_0(\bold{z})\equiv i\,(\vol(\rho_{\bold z})+i\,\cs(\rho_{\bold z}))\modulo,
\end{equation*}
where $\rho_{\bold z}$ is the boundary-parabolic representation in (\ref{rhoz}).
Furthermore, there exists a path component $\mathcal{S}_0$ of $\mathcal{S}$ satisfying
\begin{equation*}
V_0(\bold{z_{\infty}})\equiv i\,(\vol(L)+i\,\cs(L))\modulo,
\end{equation*}
for all ${\bold z}_{\infty}\in \mathcal{S}_0$.
\end{thm}

We call the value $V_0(\bold{z})$ {\it the optimistic limit of the Kashaev invariant}. Note that
it depends on the choice of the diagram and the path component $\mathcal{S}_{j}$.

\section{Proof of Theorem \ref{thm2}}\label{sec6}

This section is devoted to the proof of Theorem \ref{thm2}. Note that it was {almost} proved in \cite{Cho13b},
so we will skip several calculations and refer the results in \cite{Cho13b}.

To avoid redundant calculations, we change the definition of $W_N$ in Figure \ref{pic2} to the below:
\begin{equation}\label{W_N}
W_N:=\li(\frac{w_l}{w_m})+\li(\frac{w_l}{w_k})-\li(\frac{w_j w_l}{w_k w_m})-\li(\frac{w_m}{w_j})-\li(\frac{w_k}{w_j})
+\frac{\pi^2}{6}-\log\frac{w_j}{w_m}\log\frac{w_j}{w_k}.
\end{equation}
It is possible because, {by using $\approx$ to denote the equivalence relation defined} in Lemma 3.1 of \cite{Cho13b}, we know
$$\log\frac{w_j}{w_m}\log\frac{w_j}{w_k}\approx(\log w_j-\log w_m)(\log w_j-\log w_k)
\approx\log\frac{w_m}{w_j}\log\frac{w_k}{w_j}.$$
Therefore, changing $\log\frac{w_m}{w_j}\log\frac{w_k}{w_j}$ of $W_N$ to $\log\frac{w_j}{w_m}\log\frac{w_j}{w_k}$
does not have any effect on $\mathcal{I}$ and the optimistic limit $W_0({\bold w})$.

\begin{lem}\label{lem61}
Fix an oriented diagram $D$ of the hyperbolic link $L$, which does not have a kink.
For a solution $\bold{w}=(w_1,\ldots,w_n)\in\mathcal{T}$, 
if the variables $w_j,\ldots,w_m$ in Figure \ref{label} satisfy
\begin{equation}\label{solw}
w_j+w_l\neq w_k+w_m
\end{equation}
at all crossings, then there exists a solution $\bold{z}\in\mathcal{S}$ satisfying $\rho_{\bold w}=\rho_{\bold z}$.
Inversely, for a solution $\bold{z}=(z_1,\ldots,z_g)\in\mathcal{S}$, 
{there always exists a solution $\bold{w}\in\mathcal{T}$ satisfying $\rho_{\bold z}=\rho_{\bold w}$.}
\end{lem}

\begin{proof}
For a hyperbolic ideal octahedron in Figure \ref{fig15}, 
we assign shape parameters $t_1$, $t_2$, $t_3$, $t_4$, $u_1$, $u_2$, $u_3$ and $u_4$
to the edges CD, DA, AB, BC, CF, DE, AF and BE respectively. 
Let $u_5:=\frac{1}{u_1u_3}=\frac{1}{u_2u_4}$ be the shape parameter of the tetrahedron ABCD
assigned to the edges AC and BD.

\begin{figure}[h]
\centering
\includegraphics[scale=0.6]{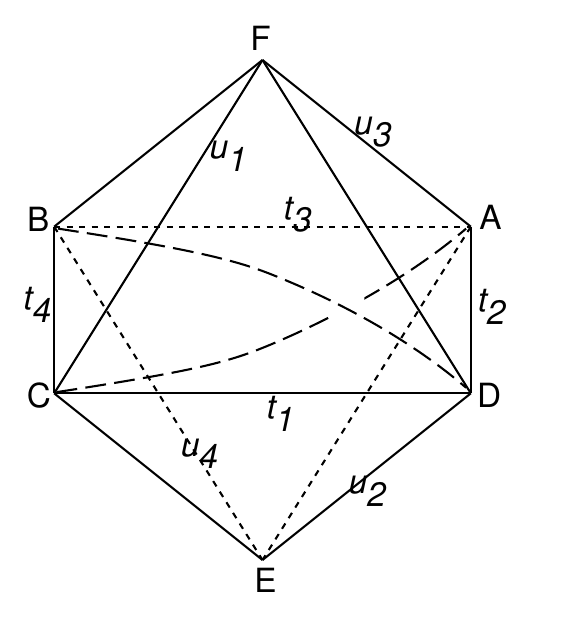}
  \caption{Assignment of variables}\label{fig15}
\end{figure}
Then we obtain the following relations.
\begin{eqnarray}
      \left\{
      \begin{array}{ll}
     t_1=u_1'' u_2'' u_5',\\
     t_2=u_2' u_3' u_5'',\\
     t_3=u_3'' u_4'' u_5',\\
     t_4=u_4' u_1' u_5'',
      \end{array}\right.\label{54term}\hspace{1cm}
      \left\{
      \begin{array}{ll}
     u_1=t_1' t_4'',\\
     u_2=t_1' t_2'',\\
     u_3=t_3' t_2'',\\
     u_4=t_3' t_4'',\\
     u_5=\left(t_1't_2''t_3't_4''\right)^{-1}.
      \end{array}\right.
\end{eqnarray}

Now we consider the octahedra placed on the crossings in Figure \ref{fig7}.
Note that the five-term triangulation and the four-term triangulation 
use the same octahedral decomposition of $\mathbb{S}^3\backslash(L\cup\{\pm\infty\})$, but the subdividing methods are different.
Therefore, if we apply (\ref{54term}) to the octahedral decomposition, we can find relations between variables $w_1,\dots,w_n$
and $z_1,\ldots,z_g$. The octahedron on Figure \ref{label}(a) (or the one in Figure \ref{fig7}(a)) gives the relations
\begin{eqnarray}
      \left\{
      \begin{array}{ll}\displaystyle
\frac{z_b}{z_a}=\left(\frac{w_m}{w_j}\right)''\left(\frac{w_k}{w_j}\right)''\left(\frac{w_j w_l}{w_k w_m}\right)',~
\frac{z_c}{z_b}=\left(\frac{w_k}{w_j}\right)'\left(\frac{w_k}{w_l}\right)'\left(\frac{w_j w_l}{w_k w_m}\right)'',
\vspace{0.2cm}\\\displaystyle
\frac{z_d}{z_c}=\left(\frac{w_k}{w_l}\right)''\left(\frac{w_m}{w_l}\right)''\left(\frac{w_j w_l}{w_k w_m}\right)',~
\frac{z_a}{z_d}=\left(\frac{w_m}{w_l}\right)'\left(\frac{w_m}{w_j}\right)'\left(\frac{w_j w_l}{w_k w_m}\right)'' ,
      \end{array}\right.\label{t1}
\end{eqnarray} and
\begin{eqnarray}
      \left\{
      \begin{array}{ll}\displaystyle
\frac{w_m}{w_j}=\left(\frac{z_b}{z_a}\right)'\left(\frac{z_a}{z_d}\right)'' ,~
\frac{w_k}{w_j}=\left(\frac{z_b}{z_a}\right)'\left(\frac{z_c}{z_b}\right)'' ,~
\frac{w_k}{w_l}=\left(\frac{z_d}{z_c}\right)'\left(\frac{z_c}{z_b}\right)'' ,\vspace{0.2cm}\\\displaystyle
\frac{w_m}{w_l}=\left(\frac{z_d}{z_c}\right)'\left(\frac{z_a}{z_d}\right)'',~
\frac{w_j w_l}{w_k w_m}=\left(\frac{z_a}{z_b}\right)''\left(\frac{z_b}{z_c}\right)'
\left(\frac{z_c}{z_d}\right)''\left(\frac{z_d}{z_a}\right)'.
      \end{array}\right.\label{u1}
\end{eqnarray} 
The octahedron on Figure \ref{label}(b) (or the one in Figure \ref{fig7}(b)) gives the relations
\begin{eqnarray}
      \left\{
      \begin{array}{ll}\displaystyle
\frac{z_b}{z_a}=\left(\frac{w_j}{w_m}\right)'\left(\frac{w_j}{w_k}\right)'\left(\frac{w_k w_m}{w_j w_l}\right)'',~
\frac{z_c}{z_b}=\left(\frac{w_j}{w_k}\right)''\left(\frac{w_l}{w_k}\right)''\left(\frac{w_k w_m}{w_j w_l}\right)',
\vspace{0.2cm}\\\displaystyle
\frac{z_d}{z_c}=\left(\frac{w_l}{w_k}\right)'\left(\frac{w_l}{w_m}\right)'\left(\frac{w_k w_m}{w_j w_l}\right)'',~
\frac{z_a}{z_d}=\left(\frac{w_l}{w_m}\right)''\left(\frac{w_j}{w_m}\right)''\left(\frac{w_k w_m}{w_j w_l}\right)',
      \end{array}\right.\label{t2}
\end{eqnarray} and
\begin{eqnarray}
      \left\{
      \begin{array}{ll}\displaystyle
\frac{w_j}{w_m}=\left(\frac{z_a}{z_d}\right)'\left(\frac{z_b}{z_a}\right)'',~
\frac{w_j}{w_k}=\left(\frac{z_c}{z_b}\right)'\left(\frac{z_b}{z_a}\right)'',~
\frac{w_l}{w_k}=\left(\frac{z_c}{z_b}\right)'\left(\frac{z_d}{z_c}\right)'', \vspace{0.2cm}\\\displaystyle
\frac{w_l}{w_m}=\left(\frac{z_a}{z_d}\right)'\left(\frac{z_d}{z_c}\right)'',~
\frac{w_k w_m}{w_j w_l}=\left(\frac{z_a}{z_b}\right)'\left(\frac{z_b}{z_c}\right)''
\left(\frac{z_c}{z_d}\right)'\left(\frac{z_d}{z_a}\right)''.
      \end{array}\right.\label{u2}
\end{eqnarray}

If $w_j,\ldots,w_m$ of each crossing is fixed, then we can determine $z_a,\ldots,z_d$ using (\ref{u1}) and (\ref{u2}),
and the inverse can be done using (\ref{t1}) and (\ref{t2}). 
Furthermore, if we consider ${\bold w}\in\mathbb{CP}^{n-1}$ and ${\bold z}\in\mathbb{CP}^{g-1}$, then $\bold w$ determines
$\bold z$ uniquely, and vice versa. 

For the set of equations
\begin{eqnarray}\label{e1}
      \left\{
      \begin{array}{ll}\displaystyle
\left(\frac{w_m}{w_j}\right)''\left(\frac{w_k}{w_j}\right)''\left(\frac{w_j w_l}{w_k w_m}\right)'\neq1,~
\left(\frac{w_k}{w_j}\right)'\left(\frac{w_k}{w_l}\right)'\left(\frac{w_j w_l}{w_k w_m}\right)''\neq1,
\vspace{0.2cm}\\\displaystyle
\left(\frac{w_k}{w_l}\right)''\left(\frac{w_m}{w_l}\right)''\left(\frac{w_j w_l}{w_k w_m}\right)'\neq1,~
\left(\frac{w_m}{w_l}\right)'\left(\frac{w_m}{w_j}\right)'\left(\frac{w_j w_l}{w_k w_m}\right)''\neq1 ,
      \end{array}\right.
\end{eqnarray}
in (\ref{t1}) and 
\begin{eqnarray}\label{e2}
      \left\{
      \begin{array}{ll}\displaystyle
\left(\frac{w_j}{w_m}\right)'\left(\frac{w_j}{w_k}\right)'\left(\frac{w_k w_m}{w_j w_l}\right)''\neq1,~
\left(\frac{w_j}{w_k}\right)''\left(\frac{w_l}{w_k}\right)''\left(\frac{w_k w_m}{w_j w_l}\right)'\neq1,
\vspace{0.2cm}\\\displaystyle
\left(\frac{w_l}{w_k}\right)'\left(\frac{w_l}{w_m}\right)'\left(\frac{w_k w_m}{w_j w_l}\right)''\neq1,~
\left(\frac{w_l}{w_m}\right)''\left(\frac{w_j}{w_m}\right)''\left(\frac{w_k w_m}{w_j w_l}\right)'\neq1,
      \end{array}\right.
\end{eqnarray}
in (\ref{t2}), direct calculation shows (\ref{e1}), (\ref{e2}) and (\ref{solw}) are equivalent each other. 
Therefore, (\ref{solw}) guarantees the determined $\bold z$ is a solution ${\bold z}\in\mathcal{S}$.

Also, for the set of equations
\begin{eqnarray}\label{e3}
      \left\{
      \begin{array}{ll}\displaystyle
\left(\frac{z_b}{z_a}\right)'\left(\frac{z_a}{z_d}\right)''\neq 1,~
\left(\frac{z_b}{z_a}\right)'\left(\frac{z_c}{z_b}\right)''\neq 1,~
\left(\frac{z_d}{z_c}\right)'\left(\frac{z_c}{z_b}\right)''\neq 1,\\\displaystyle
\left(\frac{z_d}{z_c}\right)'\left(\frac{z_a}{z_d}\right)''\neq 1,~
\left(\frac{z_a}{z_b}\right)''\left(\frac{z_b}{z_c}\right)'
\left(\frac{z_c}{z_d}\right)''\left(\frac{z_d}{z_a}\right)'\neq 1,
\end{array}\right.
\end{eqnarray}
in (\ref{u1}) and 
\begin{eqnarray}\label{e4}
      \left\{
      \begin{array}{ll}\displaystyle
\left(\frac{z_a}{z_d}\right)'\left(\frac{z_b}{z_a}\right)''\neq 1,~
\left(\frac{z_c}{z_b}\right)'\left(\frac{z_b}{z_a}\right)''\neq 1,~
\left(\frac{z_c}{z_b}\right)'\left(\frac{z_d}{z_c}\right)''\neq 1,\\\displaystyle
\left(\frac{z_a}{z_d}\right)'\left(\frac{z_d}{z_c}\right)''\neq 1,~
\left(\frac{z_a}{z_b}\right)'\left(\frac{z_b}{z_c}\right)''
\left(\frac{z_c}{z_d}\right)'\left(\frac{z_d}{z_a}\right)''\neq 1
\end{array}\right.
\end{eqnarray}
in (\ref{u2}), direct calculation shows (\ref{e3}), (\ref{e4}) { and $z_a\neq z_c, ~z_b\neq z_d$
are equivalent each other. The latter is the assumption of the solution,
hence any $\bold z\in\mathcal{S}$ determines a solution $\bold w\in\mathcal{T}$.}

Finally, if $\bold z$ and $\bold w$ are related as above, then they determine the same octahedral decomposition
and the same developing map. Therefore, we conclude $\rho_{\bold z}=\rho_{\bold w}$.

\end{proof}

Let $D(z):=\imaginary\li(z)+\log\vert z\vert\arg(1-z)$ be the Bloch-Wigner function for $z\in\mathbb{C}\backslash\{0,1\}$.
It is a well-known fact that $D(z)=\vol(T_z)$, where $T_z$ is the hyperbolic ideal tetrahedron with the shape parameter $z$.
Therefore, from Figure \ref{fig15}, we obtain
\begin{equation}\label{vol}
D(t_1)+D(t_2)+D(t_3)+D(t_4)=D(u_1)+D(u_2)+D(u_3)+D(u_4)+D(u_5).
\end{equation}

Note that the variables $t_1,\ldots,t_4,u_1,\ldots,u_5$ satisfying (\ref{54term}) determine 
a hyperbolic ideal octahedron in Figure \ref{fig15}, so (\ref{54term}) guarantees (\ref{vol}).

\begin{lem}\label{lem62}
Let $t_1, t_2, t_3, t_4, u_1, u_2, u_3, u_4, u_5\notin\{0,1,\infty\}$ be the shape parameters defined in the hyperbolic octahedron
in Figure \ref{fig15}, which satisfies (\ref{54term}) and (\ref{vol}).
Then the following identities hold for any choice of log-branch modulo $4\pi^2$.

\begin{eqnarray*}
\lefteqn{\li(t_1)-\li(\frac{1}{t_2})+\li(t_3)-\li(\frac{1}{t_4})}\\
\lefteqn{~\equiv\li(u_1)+\li(u_2)-\li(\frac{1}{u_3})-\li(\frac{1}{u_4})+\li(u_5)-\frac{\pi^2}{6}+\log u_1\log u_2}\\
  &&-\left(-\log(1-t_1)+\log(1-\frac{1}{t_4})\right)\log u_2-\left(-\log(1-t_1)+\log(1-\frac{1}{t_2})\right)\log u_1\\
  &&+\left(-\log(1-t_1)+\log(1-\frac{1}{t_4})\right)\log(1-{u_1})
  +\left(-\log(1-t_1)+\log(1-\frac{1}{t_2})\right)\log(1-u_2)\\
  &&+\left(-\log(1-t_3)+\log(1-\frac{1}{t_2})\right)\log(1-\frac{1}{u_3})
  +\left(-\log(1-t_3)+\log(1-\frac{1}{t_4})\right)\log(1-\frac{1}{u_4})\\
  &&+\left(\log(1-t_1)-\log(1-\frac{1}{t_2})+\log(1-t_3)-\log(1-\frac{1}{t_4})\right)\log(1-{u_5})\end{eqnarray*}
\begin{eqnarray*}  
\lefteqn{~\equiv\li(u_1)-\li(\frac{1}{u_2})-\li(\frac{1}{u_3})+\li({u_4})-\li(\frac{1}{u_5})+\frac{\pi^2}{6}-\log u_2\log u_3}\\
  &&+\left(-\log(1-t_3)+\log(1-\frac{1}{t_2})\right)\log u_2+\left(-\log(1-t_1)+\log(1-\frac{1}{t_2})\right)\log u_3\\
  &&+\left(-\log(1-t_1)+\log(1-\frac{1}{t_4})\right)\log(1-{u_1})
  +\left(-\log(1-t_1)+\log(1-\frac{1}{t_2})\right)\log(1-\frac{1}{u_2})\\
  &&+\left(-\log(1-t_3)+\log(1-\frac{1}{t_2})\right)\log(1-\frac{1}{u_3})
  +\left(-\log(1-t_3)+\log(1-\frac{1}{t_4})\right)\log(1-{u_4})\\
  &&+\left(\log(1-t_1)-\log(1-\frac{1}{t_2})+\log(1-t_3)-\log(1-\frac{1}{t_4})\right)\log(1-\frac{1}{u_5})~~~\modulos.
\end{eqnarray*}

\end{lem}

\begin{proof} See the proof of Lemma 5.1 in \cite{Cho13b}.

\end{proof}

Let $\bold w\in\mathcal{T}$ and $\bold z\in\mathcal{S}$ be the corresponding pair in Lemma \ref{lem61}.
To prove 
\begin{equation}\label{goal2}
V_0(\bold{z})\equiv W_0(\bold{w})\modulos,
\end{equation} 
we consider the two cases of the crossing with parameters $z_a,\ldots,z_d,w_j,\ldots,w_m$ in Figure \ref{label}.

For the case of Figure \ref{label}(a), 
we let $t_1=\frac{z_b}{z_a}$, $t_2=\frac{z_c}{z_b}$, $t_3=\frac{z_d}{z_c}$, $t_4=\frac{z_a}{z_d}$,
$u_1=\frac{w_m}{w_j}$, $u_2=\frac{w_k}{w_j}$, $u_3=\frac{w_k}{w_l}$, $u_4=\frac{w_m}{w_l}$ and $u_5=\frac{w_j w_l}{w_k w_m}$
so that (\ref{54term}) satisfies. 
Then the potential function of a crossing defined in Figure \ref{potentialV} is expressed by
$$V_P(z_a,\ldots,z_d):=\li(t_1)-\li(\frac{1}{t_2})+\li(t_3)-\li(\frac{1}{t_4}),$$
and the potential function of a positive crossing defined in Figure \ref{pic2}(a) is expressed by
\begin{eqnarray*}
\lefteqn{W_P(w_j,w_k,w_l,w_m)}\\
&&=\li(u_1)+\li(u_2)-\li(\frac{1}{u_3})-\li(\frac{1}{u_4})+\li(u_5)-\frac{\pi^2}{6}+\log u_1\log u_2.
\end{eqnarray*}
Using Lemma \ref{lem62}, we can calculate
\begin{eqnarray}
V_{P0}-W_{P0}&\equiv&-(\log w_j-\log w_m)\log z_a-(\log w_k-\log w_j)\log z_b\nonumber\\
&&+(\log w_k-\log w_l)\log z_c+(\log w_l-\log w_m)\log z_d\modulos.\label{eq44}
\end{eqnarray}
(The details are in (41)--(42) and the following paragraphs of Section 5 in \cite{Cho13b}.
Note that, in \cite{Cho13b}, we denoted $V_P$ and $W_P$ by $X(z_a,\ldots,z_d)$ and $P_1(w_j,\ldots,w_m)$ respectively.)

For the case of Figure \ref{label}(b), 
we let $t_1=\frac{z_a}{z_d}$, $t_2=\frac{z_b}{z_a}$, $t_3=\frac{z_c}{z_b}$, $t_4=\frac{z_d}{z_c}$,  
$u_1=\frac{w_l}{w_m}$, $u_2=\frac{w_j}{w_m}$, $u_3=\frac{w_j}{w_k}$, $u_4=\frac{w_l}{w_k}$ and $u_5=\frac{w_k w_m}{w_j w_l}$
so that (\ref{54term}) satisfies. 
Then the potential function of a crossing defined in Figure \ref{potentialV} is expressed by
$$V_N(z_a,\ldots,z_d):=\li(t_1)-\li(\frac{1}{t_2})+\li(t_3)-\li(\frac{1}{t_4}),$$
and the potential function of a negative crossing defined in (\ref{W_N}) is expressed by
\begin{eqnarray*}
\lefteqn{W_N(w_j,w_k,w_l,w_m)}\\
&&=\li(u_1)-\li(\frac{1}{u_2})-\li(\frac{1}{u_3})+\li({u_4})-\li(\frac{1}{u_5})+\frac{\pi^2}{6}-\log u_2\log u_3.
\end{eqnarray*}
Using Lemma \ref{lem62}, we can calculate
\begin{eqnarray}
V_{N0}-W_{N0}&\equiv&-(\log w_j-\log w_m)\log z_a-(\log w_k-\log w_j)\log z_b\nonumber\\
&&+(\log w_k-\log w_l)\log z_c+(\log w_l-\log w_m)\log z_d\modulos.\label{eq45}
\end{eqnarray}

Note that the right-hand sides of (\ref{eq44}) and (\ref{eq45}) coincide. We can deduce the general rule of
these equations using Figure \ref{final}.

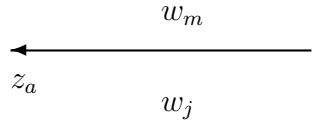
\begin{figure}[h]
\centering
\begin{picture}(4,2)\thicklines
   \put(4,1){\vector(-1,0){4}}
   \put(2,1.4){$w_m$}
   \put(2,0.2){$w_j$}
   \put(0,0.5){$z_a$}
  \end{picture}
  \caption{Side assigned by $z_a$}\label{final}
\end{figure}

For the side with $z_a$ in Figure \ref{final}, when it goes out of a crossing, 
the contribution to (\ref{eq44}) or (\ref{eq45}) of the crossing is
$$-(\log w_j-\log w_m)\log z_a,$$
and when it goes into a crossing, the contribution is
$$+(\log w_j-\log w_m)\log z_a.$$
Therefore, if we consider the whole crossings of the link diagram, 
the right-hand sides of (\ref{eq44}) or (\ref{eq45}) at all crossings are cancelled out
and we obtain (\ref{goal2}). This completes the proof of Theorem \ref{thm2}.

\section{Example of the twist knots}\label{sec7}

In this section, we apply Theorem \ref{thm2} to the example of the twist knot in Section 6 of \cite{Cho13a}
and show several numerical results. For the calculations, we assume the principal branch of logarithm.
Also we use the definition of $W_N$ in Figure \ref{pic2}(b).

Let $T_n$ ($n\geq 1$) be the twist knot with $n+3$ crossings in Figure \ref{twistknots}. 
For example, $T_1$ is the figure-eight knot $4_1$ and $T_2$ is the $5_2$ knot.
We follow the orientations in Figure \ref{twistknots}.

\begin{figure}[h]\centering
\subfigure[$n$ is odd]{\includegraphics[scale=0.85]{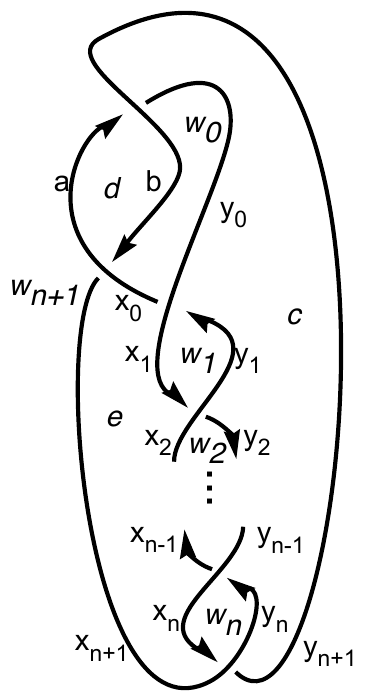}}\hspace{1cm}
\subfigure[$n$ is even]{\includegraphics[scale=0.85]{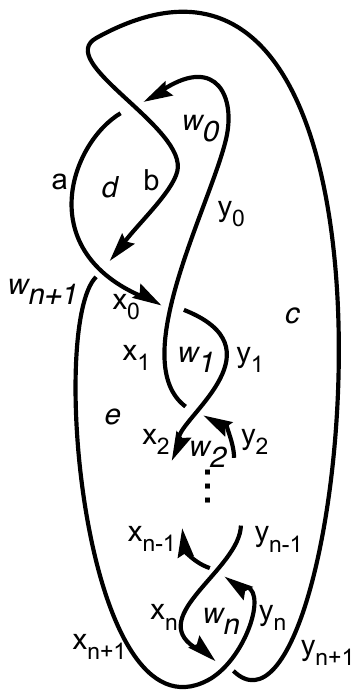}}
 \caption{Twist knot $T_n$}\label{twistknots}
\end{figure}

We assign variables $a, b, x_0,\ldots,x_{n+1},y_0,\ldots,y_{n+1}$ to the sides 
and $c, d, e, w_0,\ldots, w_{n+1}$ to the regions of Figure \ref{twistknots} respectively.
Let
\begin{eqnarray*}
A_k:=\li(\frac{c}{w_k})+\li(\frac{c}{w_{k+1}})-\li(\frac{c\,e}{w_k w_{k+1}})-\li(\frac{w_k}{e})-\li(\frac{w_{k+1}}{e})
+\frac{\pi^2}{6}-\log\frac{w_{k}}{e}\log\frac{w_{k+1}}{e},\\
B_k:=\li(\frac{e}{w_k})+\li(\frac{e}{w_{k+1}})-\li(\frac{c\,e}{w_k w_{k+1}})-\li(\frac{w_k}{c})-\li(\frac{w_{k+1}}{c})
+\frac{\pi^2}{6}-\log\frac{w_{k}}{c}\log\frac{w_{k+1}}{c},
\end{eqnarray*}
for $k=0,1,\ldots,n$. If $n$ is odd, the potential function $W(T_n;c,d,e,w_0,\ldots,w_{n+1})$ of Figure \ref{twistknots}(a) is
\begin{eqnarray*}
\lefteqn{W(T_n;c,d,e,w_0,\ldots,w_{n+1})}\\
&&=\left\{-\li(\frac{w_{n+1}}{c})-\li(\frac{w_{n+1}}{d})+\li(\frac{w_0 w_{n+1}}{c\,d})
+\li(\frac{c}{w_0})+\li(\frac{d}{w_0})\right.\\
&&~~~~\left.-\frac{\pi^2}{6}+\log\frac{c}{w_0}\log\frac{d}{w_0}\right\}\\
&&~~+\left\{-\li(\frac{w_0}{d})-\li(\frac{w_0}{e})+\li(\frac{w_0 w_{n+1}}{d\,e})
+\li(\frac{d}{w_{n+1}})+\li(\frac{e}{w_{n+1}})\right.\\
&&\left.~~~~-\frac{\pi^2}{6}+\log\frac{d}{w_{n+1}}\log\frac{e}{w_{n+1}}\right\}\\
&&~~+\sum_{k=0}^{(n-1)/2}\left(A_{2k}+B_{2k+1}\right),
\end{eqnarray*}
and if $n$ is even, the potential function $W(T_n;c,d,e,w_0,\ldots,w_{n+1})$ of Figure \ref{twistknots}(b) is
\begin{eqnarray*}
\lefteqn{W(T_n;c,d,e,w_0,\ldots,w_{n+1})}\\
&&=\left\{\li(\frac{c}{w_{0}})+\li(\frac{c}{w_{n+1}})-\li(\frac{c\,d}{w_0 w_{n+1}})
-\li(\frac{w_0}{d})-\li(\frac{w_{n+1}}{d})\right.\\
&&~~~~~\left.+\frac{\pi^2}{6}-\log\frac{w_{0}}{d}\log\frac{w_{n+1}}{d}\right\}\\
&&~~+\left\{\li(\frac{d}{w_{0}})+\li(\frac{d}{w_{n+1}})-\li(\frac{d\,e}{w_0 w_{n+1}})
-\li(\frac{w_0}{e})-\li(\frac{w_{n+1}}{e})\right.\\
&&~~~~~\left.+\frac{\pi^2}{6}-\log\frac{w_{0}}{e}\log\frac{w_{n+1}}{e}\right\}\\
&&~~+B_0+\sum_{k=1}^{n/2}\left(A_{2k-1}+B_{2k}\right).
\end{eqnarray*}

In Section 6 of \cite{Cho13a}, we chose $(a,b,x_0,\ldots,x_{n+1},y_0,\dots, y_{n+1})$ by
\begin{eqnarray*}
a=2,~b=-1,~x_0=t, ~y_0=1+\frac{2}{t},~
x_1=\frac{t(t+2)}{t^2-4t+8},~y_1=\frac{4}{t},\\
x_{k+1}=\frac{x_k y_k}{-x_{k-1}+x_k+y_k},~y_{k+1}=x_k+y_k-\frac{x_k y_k}{y_{k-1}},~x_{n+1}=3,~y_{n+1}=1,
\end{eqnarray*}
where $k=1,\ldots,n-1$, and $t$ is a solution of the defining equation in Table \ref{table1}.
All the solutions $t$ of the defining equation determine the solutions in $\mathcal{S}$
and the corresponding representation
$$\rho(T_n)(t):\pi_1(\mathbb{S}^3\backslash T_n)\longrightarrow{\rm PSL}(2,\mathbb{C}).$$

\begin{table}[h]
{\begin{tabular}{|c|c|}\hline
  $n$ & Defining equation of $t$\\\hline
  1 & $16-12t+3t^2=0$\\\hline
  2 & $-64+80t-40t^2+7t^3=0$\\\hline
  3 & $256-448t+336t^2-120t^3+17t^4=0$\\\hline
  4 & $-2048+4608t-4608t^2+2464t^3-696t^4+82t^5=0$\\\hline
  5 & $4096-11264t+14080t^2-9984t^3+4192t^4-980t^5+99t^6=0$\\\hline
\end{tabular}\caption{Defining equation of $t$ for $n=1,\dots,5$}\label{table1}}
\end{table}

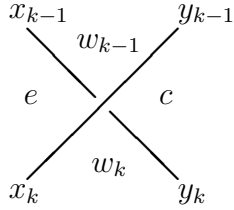
\begin{figure}[h]
\centering
\begin{picture}(6,5)  
  \setlength{\unitlength}{0.5cm}\thicklines
    \put(6,5){\line(-1,-1){4}}
    \put(2,5){\line(1,-1){1.8}}
    \put(4.2,2.8){\line(1,-1){1.8}}
    \put(3.7,1.2){$w_k$}
    \put(5.5,3){$c$}
    \put(3.3,4.5){$w_{k-1}$}
    \put(1.9,3){$e$}
    \put(1.5,5.3){$x_{k-1}$}
    \put(6,5.3){$y_{k-1}$}
    \put(1.5,0.5){$x_k$}
    \put(6,0.5){$y_k$}
  \end{picture}\caption{The ($k+2$)-th crossing for $k=1,\ldots,n+1$}\label{crossing k+2}
    \end{figure}

Using the equations (\ref{u1}) and (\ref{u2}), we can express $(c,d,e,w_0,\ldots,w_{n+1})$ in terms of $t$. 
Specifically, the ($k+2$)-th crossing (in the order from top to bottom) in Figure \ref{twistknots} becomes Figure \ref{crossing k+2}
and it determines
$$\frac{e}{w_k}=\left(\frac{y_k}{x_k}\right)'\left(\frac{x_k}{x_{k-1}}\right)'',$$
for $k=1,\ldots,n+1$. The first crossing in Figure \ref{twistknots} gives an equation of $c$
$$\frac{c}{w_{n+1}}=\left(\frac{a}{y_{n+1}}\right)'\left(\frac{y_{n+1}}{y_0}\right)''=\frac{2}{t}.$$
The second crossing in Figure \ref{twistknots} gives more simple equation of $w_{n+1}$
$$
\frac{e}{w_{n+1}}=\left(\frac{x_{n+1}}{a}\right)'\left(\frac{x_0}{x_{n+1}}\right)''=-\frac{2(t-3)}{t},
$$ and other equations of $d$ and $w_0$
$$\frac{d}{w_{n+1}}=\left(\frac{x_{n+1}}{a}\right)'\left(\frac{a}{b}\right)''=-3,
~\frac{e}{w_0}=\left(\frac{b}{x_0}\right)'\left(\frac{x_0}{x_{n+1}}\right)''=\frac{t-3}{t+1}.$$

Therefore, after choosing $e=1$, we can express $(c,d,e,w_0,\ldots,w_{n+1})$ in terms of $t$ by
\begin{eqnarray*}
c=-\frac{1}{t-3},~d=\frac{3t}{2(t-3)},~e=1,~w_0=\frac{t+1}{t-3},~
w_k=\left(\frac{x_k}{y_k}\right)''\left(\frac{x_{k-1}}{x_k}\right)',
\end{eqnarray*}
for $k=1,\ldots,n+1$. The exact expression of $w_k$ for $k=1,\ldots,6$ is in Table \ref{table2}.

\begin{table}[h]
{\begin{tabular}{|c|c|}\hline
  $k$ & $w_k$\\\hline
  0 & $(1+t)/(-3+t)$\\\hline
  1 & $-(16+t^2)/((-3+t)t^2)$\\\hline
  2 & $(256-256t+112t^2 -16t^3 -3t^4 +t^5)/((-3+t)t^4)$\\\hline
  3 & $(-4096+8192t-7424t^2 +3584t^3 -864t^4 +32t^5 +27t^6 -4t^7)/((-3+t)t^6)$\\\hline
  4 &$(65536-196608t+274432t^2 -225280t^3 +115456t^4 -35584t^5 +5152t^6$\\
     &$+320t^7 -231t^8 +25t^9)/((-3+t)t^8)$\\\hline
  5 & $(-1048576+4194304t-7929856t^2 + 9175040t^3 -7094272t^4 +3760128t^5$\\
     &$-1337088t^6+287232t^7 -21232t^8 -6048t^9 +1751t^{10} -144t^{11})/((-3+t)t^{10})$\\\hline
  6 &$(16777216-83886080t+200278016t^2 -298844160t^3 +307822592t^4$\\
     &$-228524032t^5 +123846656t^6 - 48324608t^7 +12842496t^8 -1930752t^9 $\\
     &$- 2544t^{10} +66288t^{11} -12587t^{12} +841t^{13})/((-3+t)t^{12})$\\\hline
\end{tabular}\caption{Expressions of $w_k$ in terms of $t$ for $k=1,\ldots,6$}\label{table2}}
\end{table}

For the solutions $t$ of the defining equations, the numerical values of the corresponding optimistic limits 
$$W_0(T_n)(t)\equiv i(\vol(\rho(T_n)(t))+i\,\cs(\rho(T_n)(t)))\modulo,$$
for $n=1,\ldots,5,$ are in Table \ref{table3}. Note that these values exactly coincide with 
the optimistic limits of Kashaev invariants in Table 3 of \cite{Cho13a}.

\begin{table}[h]
{\begin{tabular}{|c|l|l|}\hline
  $n$ &\hspace{2cm} $t$ &$W_0(T_n)(t)\equiv i(\vol(\rho(T_n)(t))+i\,\cs(\rho(T_n)(t)))$\\\hline
  1 & $t=2+1.1547...i$ & $i(2.0299...+0\,i)$\\
    &$t=2-1.1547...i$ & $i(-2.0299...+0\,i)$\\\hline
  2 & $t=1.4587...+1.0682...i$ & $i(2.8281...+3.0241...i)$\\
    & $t=1.4587...-1.0682...i$ & $i(-2.8281...+3.0241...i)$\\
    & $t=2.7969...$ &$i(0-1.1135...i)$\\\hline
  3 & $t=1.2631...+1.0347...i$ & $i(3.1640...+6.7907...i)$\\
    & $t=1.2631...-1.0347...i$ & $i(-3.1640...+6.7907...i)$\\
    & $t=2.2664...+0.7158...i$ & $i(1.4151...+0.2110...i)$\\
    & $t=2.2664...-0.7158...i$ &$i(-1.4151...+0.2110...i)$\\\hline
  4 & $t=1.1713...+1.0202...i$ & $i(3.3317...+10.9583...i)$\\
    & $t=1.1713...-1.0202...i$ & $i(-3.3317...+10.9583...i)$\\
    & $t=1.8097...+0.9073...i$ &$i(2.2140...+1.8198...i)$\\
    & $t=1.8097...-0.9073...i$ & $i(-2.2140...+1.8198...i)$\\
    & $t=2.5257...$ & $i(0-0.8822...i)$\\\hline
  5 & $t=1.1208...+1.0129...i$ & $i(3.4272...+15.3545...i)$\\
    & $t=1.1208...-1.0129...i$ & $i(-3.4272...+15.3545...i)$\\
    & $t=1.5498...+0.9676...i$ & $i(2.6560...+4.6428...i)$\\
    & $t=1.5498...-0.9676...i$ &$i(-2.6560...+4.6428...i)$\\
    & $t=2.2789...+0.4876...i$ & $i(1.1087...-0.2581...i)$\\
    & $t=2.2789...-0.4876...i$ & $i(-1.1087...-0.2581...i)$\\\hline
\end{tabular}\caption{Values of $W_0(T_n)(t)$ for $n=1,\ldots,5$}\label{table3}}
\end{table}

\appendix
\section{Change on the signs of the variables}

In this appendix, we show that the change on the signs of the variables of the potential function does not have an effect on 
{the set of equations $\mathcal{I}$ and the optimistic limit. Note that this property will be used in the author's later article.}

Let $W(w_1,\ldots,w_n)$ be the potential function defined in Section \ref{sec2}.
Let $\tau_1,\ldots,\tau_n,\epsilon_1,\ldots,\epsilon_n\in\{-1,1\}$ be fixed signs and define another potential function

\begin{equation*}
\widetilde{W}(w_1,\ldots,w_m):=W(\tau_1 w_1^{\epsilon_1},\ldots,\tau_n w_n^{\epsilon_n}).
\end{equation*}
In the same way, we define 
\begin{equation*}
\widetilde{\mathcal{I}}:=\left\{\left.\exp\left(w_k\frac{\partial \widetilde{W}}{\partial w_k}\right)=1\right|k=1,\ldots,n\right\}
\end{equation*}
and $\widetilde{\mathcal{T}}$ be the solution set of $\widetilde{\mathcal{I}}$.
Also, for ${\bold w}=(w_1,\ldots,w_n)$, define
$$\widetilde{\bold w}:=(\tau_1 w_1^{\epsilon_1},\ldots,\tau_n w_n^{\epsilon_n}).$$

\begin{pro} There is a one-to-one correspondence between ${\bold w}\in{\mathcal{T}}$ and 
$\widetilde{\bold w}\in\widetilde{\mathcal{T}}$.
Furthermore, we have
\begin{equation}\label{eq45}
\widetilde{W}_0(\widetilde{\bold w})\equiv W_0({\bold w})~~({\rm mod}~2\pi^2).
\end{equation}
\end{pro}

\begin{proof} At first, we show
$\widetilde{\bold w}\in\widetilde{\mathcal{T}}$ for each ${\bold w}\in{\mathcal{T}}$. Note that
$$w_k\frac{\partial }{\partial w_k}\li\left(\frac{\tau_k w_k^{\epsilon_k}}{\tau_j w_j^{\epsilon_j}}\right)=
\epsilon_k\cdot \log\left(1-\frac{\tau_k w_k^{\epsilon_k}}{\tau_j w_j^{\epsilon_j}}\right)$$
implies
\begin{equation*}
w_k\frac{\partial }{\partial w_k}\li\left(\frac{\tau_k w_k^{\epsilon_k}}{\tau_j w_j^{\epsilon_j}}\right)\Big|_{{\bold w}=\widetilde{\bold w}}
=\epsilon_k\cdot\log\left(1-\frac{w_k}{w_j}\right)
=\epsilon_k\cdot w_k\frac{\partial }{\partial w_k}\li\left(\frac{w_k}{w_j}\right),
\end{equation*}
where $|_{{\bold w}=\widetilde{\bold w}}$ means the evaluation of the equation at $\widetilde{\bold w}$.
Using these kinds of calculations, we obtain
\begin{equation}\label{eq46}
w_k\frac{\partial \widetilde{W}}{\partial w_k}\Big|_{{\bold w}=\widetilde{\bold w}}=\epsilon_k\cdot w_k\frac{\partial {W}}{\partial w_k},
\end{equation}
which shows $\widetilde{\bold w}\in\widetilde{\mathcal{T}}$ and the coincidence of $\mathcal{I}$ and $\widetilde{\mathcal{I}}$.
Therefore there is a one-to-one correspondence between $\mathcal{T}$ and $\widetilde{\mathcal{T}}$.

Note that $\widetilde{W}(\widetilde{\bold w})=W({\bold w})$ holds trivially.
For ${\bold w}\in{\mathcal{T}}$, the value of (\ref{eq46}) is zero module $2\pi i$. 
Therefore, (\ref{eq45}) follows from
\begin{equation*}
(w_k\frac{\partial \widetilde{W}}{\partial w_k})\log w_k\Big|_{{\bold w}=\widetilde{\bold w}}
\equiv\epsilon_k\cdot(w_k\frac{\partial {W}}{\partial w_k})\log (\tau_k w_k^{\epsilon_k})
\equiv(w_k\frac{\partial {W}}{\partial w_k})\log w_k~~({\rm mod}~2\pi^2).
\end{equation*}

\end{proof}

We finally remark that the same result holds for the potential function
of the Kashaev invariant in Section \ref{sec5} by the exactly same arguments. 

\vspace{5mm}
\begin{ack}
  He appreciates Yuichi Kabaya, Hyuk Kim and Seonhwa Kim for discussions and suggestions on this work.
\end{ack}

\bibliography{VolConj}
\bibliographystyle{abbrv}

{
\begin{flushleft}
  Department of Mathematics, Pohang Mathematics Institute(PMI), Gyungbuk 790-784, Republic of Korea\\
  \vspace{0.4cm}
E-mail: dol0425@gmail.com\\
\end{flushleft}}
\end{document}